\documentclass[Threeside,11pt]{article} 
\usepackage{etoolbox}
\usepackage{setspace}
\usepackage{amsmath,amsthm,amsfonts,amssymb,amscd,mathrsfs}
\usepackage{tikz}
\usepackage{tikz-cd}
\usepackage{float}
\usepackage{subcaption}
\usepackage[latin1]{inputenc}
\usepackage[english]{babel}
\usepackage{algorithm}
\usepackage{adjustbox}
\usepackage{cite}
\usepackage{graphicx}
\usepackage{colortbl,dcolumn}
\usepackage{marvosym}
\usepackage{ifsym}
\usepackage{paralist}
\usepackage{xcolor}
\usepackage{array}
\usepackage{geometry}
\usepackage{enumitem}
\graphicspath{{figures/}}
\usetikzlibrary{arrows}

\definecolor{myPurple}{RGB}{109,90,207}
\definecolor{myGreen}{RGB}{4,130,67}
\definecolor{myGold}{RGB}{253,177,71}
\definecolor{myBurgundy}{RGB}{63,1,44}
\definecolor{myTeal}{RGB}{8,119,127}
\definecolor{myCream}{RGB}{255,216,177}
\definecolor{myOrange}{RGB}{225,119,1}
\definecolor{mydeepgreen}{RGB}{3,100,50}

\usetikzlibrary{shapes.geometric,arrows.meta,calc,decorations.markings,bending,arrows}
\tikzset{
	ball3Donly/.style={
		circle, minimum size=1.4cm,
		shading=ball,
		ball color=myPurple,
		draw=myPurple
	}
}
\pgfmathsetmacro{\wNear}{2.0pt}
\pgfmathsetmacro{\wMid}{1.3pt}
\pgfmathsetmacro{\wFar}{0.7pt}
\pgfmathsetmacro{\bend}{6mm}

\makeatletter
\patchcmd{\thebibliography}
{\list}
{\small\setstretch{1.2}\list}
{}{}
\makeatother

\allowdisplaybreaks
\newtheorem{lemma}{\bf Lemma}[section]
\newtheorem{theorem}{\bf Theorem}[section]

\newtheorem{corollary}{\bf Corollary}[section]
\newtheorem{definition}{\bf Definition}[section]
\newtheorem{remark}{\bf Remark}[section]

\numberwithin{equation}{section}

\usepackage[
colorlinks,
citecolor=blue,
linkcolor=blue]{hyperref}

\begin{document}
\title{{\sl A generic approach via relative singularity and controllability: Frequency-preserving with arbitrarily weak regularity in parameterized Hamiltonian systems}}

\author{Zhicheng Tong\thanks{School of Mathematics, Jilin University, Changchun 130012, P. R.  China. Email: \url{tongzc25@jlu.edu.cn}} 
	\and 
	Yong Li\thanks{Corresponding author. School of Mathematics, Jilin University, Changchun 130012, P. R.  China;   Center for Mathematics and Interdisciplinary Sciences,  Northeast  Normal  University, Changchun 130024, P. R. China. Email: \url{liyong@jlu.edu.cn}}
}

\date{}
\maketitle

\begin{abstract}
In this paper, we introduce a novel and generic approach to prove the persistence of frequency-preserving invariant tori in parameterized Hamiltonian systems, addressing irregular continuity with respect to parameters. Unlike traditional methods that strongly rely on domain extraction techniques or uniform weak convexity of the frequency mapping, we propose the concepts of relative singularity and controllability for the first time. These concepts enable us to deal with a wide range of explicit parameterized Hamiltonian systems with arbitrarily weak regularity, thereby overcoming a previously insurmountable challenge. We also construct several counterexamples to highlight the indispensability of our new conditions in the sense of frequency-preserving. Furthermore, we demonstrate the broad applicability of our results to various cases with explicit arbitrarily weak regularity, including the partial frequency-preserving case and the  infinite-dimensional case without any spectral asymptotics. Overall, our approach, based on the concepts of relative singularity and controllability, illustrates its genericity in the frequency-preserving KAM theory.\vspace{2mm}
\end{abstract}

%\tableofcontents

\section{Introduction and the main KAM results}\label{FPKAMintro}
The motivation for the celebrated KAM theory arises from the stability problems in celestial mechanics. It stems from the works of Kolmogorov and Arnold \cite{MR0140699,MR0163025,MR0170705,MR0068687} and Moser \cite{MR0147741,MR0199523,MR0206461}. To be more precise, it mainly concerns the preservation of quasi-periodic and almost periodic  trajectories of Hamiltonian systems under small perturbations. The main challenge in KAM theory is dealing with small divisors, which can be effectively controlled via a linear iteration (Newtonian iteration) that admits super-exponential convergence.
So far, KAM theory has been well developed and widely applied to a variety of dynamical systems and PDEs. For some fundamental developments, we refer to Kuksin \cite{MR0911772}, Eliasson \cite{MR1001032}, P\"oschel \cite{MR1022821}, Treshch\"{e}v \cite{MR1025685}, Wayne \cite{MR1040892}, Bourgain \cite{MR1316975,MR1345016}, Kuksin and P\"oschel \cite{MR1370761}, etc. We also mention the readable survey of Sevryuk \cite{MR2078576}.
As to recent advancements and summaries in this direction, see Bambusi and Gr\'{e}bert \cite{MR2272975}, Khanin et al. \cite{MR2276445}, Berti et al. \cite{MR3112201}, Khesin et al. \cite{MR3269186}, Eliasson et al. \cite{MR3357183}, Medvedev et al. \cite{MR3366638}, Chierchia and Procesi \cite{MR4570702}, Qian et al. \cite{MR4669322}, Liu and Wang \cite{arXiv:2502.10969}, and the references therein. We would also like to mention the recent work on converse KAM theory by Cheng and Wang \cite{MR3061774}, Wang \cite{MR4385768,WLARXIV}, Sorrentino and Wang \cite{arXiv:2504.06773}, and the references therein.

\setcounter{footnote}{0}
\renewcommand{\thefootnote}{\fnsymbol{footnote}}
Focusing on the question of how much of the dynamics can be preserved under perturbations, the classic KAM theory tells us that, for an analytic, integrable Hamiltonian system, most quasi-periodic tori survive small analytic perturbations. However, the toral frequencies are generally drifted subjected to perturbations, unless the Hamiltonian system under consideration has certain nondegeneracy (e.g., the Kolmogorov nondegenerate condition in \cite{MR0163025,MR0068687}).
In the absence of such nondegeneracy, the problem of finding a KAM torus possessing the same frequency as the unperturbed one is not a simple issue, while being fundamental in KAM theory. Another instance of that same challenge is weakening the continuous dependence in parameterized settings, e.g., from the bidirectional Lipschitz regularity (e.g., see Kuksin and P\"oschel \cite{MR1370761}, Berti and Biasco \cite{MR2819413}, Gr\'{e}bert and Thomann \cite{MR2837120}, Liu and Yuan \cite{MR2842962}) to H\"older regularity or even arbitrarily weak regularity with respect to the parameter. On this aspect, by introducing certain topological degree condition and (uniform) weak convexity condition, Du et al. \cite{DL} obtained a frequency-preserving KAM theorem via H\"older continuity with respect to the parameter, through a parameter translation technique and a quasi-linear (quasi-Newtonian) KAM iterative scheme. Building upon the previous contributions, Tong et al. \cite{MR4731278} recently extended this result by assuming only modulus of continuity for Hamiltonian systems without dimensionality restrictions. Similarly, Liu et al. \cite{MR4806284} investigated mappings with the intersection property and established a frequency-preserving version of Moser's theorem. Later, Du et al. \cite{DTL}\footnote{\cite{DTL} appeared after this paper, and the infinite-dimensional KAM theorem presented therein can be significantly optimized based on the novel approach introduced in this paper, as seen in Section \ref{FPKAMINFINITE}. However, these two papers do not have an inclusive relationship with each other; rather, each has its own independent focus and innovation.} gave the detailed proof of the frequency-preserving KAM theorem for infinite-dimensional Hamiltonian systems in \cite{MR4731278} and further did not impose any spectral asymptotic conditions for  frequency mappings.
 
Back to our concern on frequency-preserving in parameterized KAM theory,  it should be noted that requiring only the topological degree condition is not sufficient, as demonstrated by the counterexample given in \cite{DL}. Indeed, under other existing techniques, the new parameter in each KAM iteration may not converge. Instead, the counterexample constructed in \cite{DL} shows that the parameter sequence may exhibit repeated oscillations, rendering the corresponding KAM scheme invalid.  However, the uniformity of the weak convexity condition, while effective, limits its applicability in a broader range of dynamical systems. Specifically, it appears challenging to construct explicit examples that possess regularity \textit{strictly weaker} than the Lipschitz type, as it is necessary to consider whether the frequencies constitute an empty set (see  \cite[Section 6]{MR4731278} for a more detailed discussion). \textit{Therefore, one of the main motivations of this paper is to eliminate this essential restriction and instead introduce a Relative Singularity Condition and a Controllability Condition.} This allows us to construct many \textit{explicit} parameterized Hamiltonian systems that admit frequency-preserving KAM tori, while still requiring only continuous dependence on the parameter. This was not achievable in existing works. One will also see later that our new assumptions proposed in Theorem \ref{FPKAMT1} are indispensable for frequency-preserving, as illustrated by certain counterexamples constructed in Section \ref{FPKAMCONTEREXAMPLES}. Moreover, we provide additional counterexamples to discuss various basic problems related to frequency-preserving. As a consequence, we provide a nearly comprehensive study of frequency-preserving KAM theory in parameterized settings.

Let us get in the main topic. Consider the parameterized family of the perturbed Hamiltonian systems
\begin{equation}\label{FPKAMHamilton}
	\left\{ \begin{gathered}
		H:G \times {\mathbb{T}^n} \times \mathcal{O} \to {\mathbb{R}^1}, \hfill \\
		H\left( {y,x,\xi,\varepsilon } \right) = \left\langle {\omega \left( \xi  \right),y} \right\rangle  + \varepsilon P\left( {y,x,\xi  } \right), \hfill \\
	\end{gathered}  \right.
\end{equation}
where $x$ is the angle variable in the standard torus $ \mathbb{T}^n = \mathbb{R}^n / 2\pi \mathbb{Z}^n $, $y$ is the action variable in $G \subset \mathbb{R}^n$, $\xi$ is a parameter in  $\mathcal{O} \subset \mathbb{R}^n$, and $n\in \mathbb{N}^+$. Here, $G$ and $\mathcal{O}$ are connected compact domains with interior points. Moreover, $\omega(\cdot)$ is continuous in $\xi$ on $\mathcal{O}$, $P(\cdot, \cdot, \xi)$ is real analytic in $y$ and $x$ on $G \times \mathbb{T}^n$, $P(y, x, \cdot)$ is only continuous in the parameter $\xi$, and $\varepsilon > 0$ is sufficiently small. Such Hamiltonian systems are usually called parameterized ones.

For the sake of convenience, we introduce the following notations. Denote by $|\cdot|$ the sup-norm for vectors in $\mathbb{R}^n$ without causing ambiguities. We formulate that in the limit process, $f_1(x) = \mathcal{O}^{\#}(f_2(x))$ means there are absolute positive constants $\ell_1$ and $\ell_2$ such that $\ell_1 f_2(x) \leqslant f_1(x) \leqslant \ell_2 f_2(x)$, and $f_1(x) = \mathcal{O}(f_2(x))$ (or $f_1(x) \lesssim f_2(x)$) implies that there exists an absolute positive constant $\ell_3$ such that $|f_1(x)| \leqslant \ell_3 f_2(x)$ (or $f_1(x) \leqslant \ell_3 f_2(x)$). Finally, $f_1(x) = o(f_2(x))$ is equivalent to $\lim f_1(x) / f_2(x) = 0$.

Let us start with the following assumptions:

\begin{enumerate}[label=(H\arabic*), ref=(H\arabic*)]
	\item \label{aph1} [Internal Condition] For $ \Upsilon  \in \left(\omega(\mathcal{O})\right)^o$ given in advance, there exists $ \xi_0 \in \mathcal{O}^o$ such that $ \Upsilon  = \omega \left( {{\xi _0}} \right) $ admits Diophantine nonresonance, which is defined as:
	\begin{equation}\label{FPKAMDIO}
		\left| {\left\langle {k,\omega \left( {{\xi _0}} \right)} \right\rangle } \right| \geqslant \gamma {\left| k \right|^{ - \tau }}, \quad \gamma  > 0, \quad \tau>\max\{n-1,1\}, \quad \forall 0 \ne k \in {\mathbb{Z}^n}.
	\end{equation}

	\item \label{aph2} [Relative Singularity Condition] There exists a neighborhood $ \mathcal{V} \subset \mathcal{O} $ of $ \xi_0 $, such that the following holds (allowing the supremum to be continuously supplemented according to the sup-limit)
	\[\mathop {\sup }\limits_{\xi  \ne \zeta ,\xi ,\zeta  \in \mathcal{V}} \frac{{\mathop {\sup }\limits_{y,x \in G \times {\mathbb{T}^n}} \left| {P\left( {y,x,\xi } \right) - P\left( {y,x,\zeta } \right)} \right|}}{{\left| {\omega \left( \xi  \right) - \omega \left( \zeta  \right)} \right|}} <+\infty. \]
	
\item \label{aph3} [Controllability Condition] Assume that $ \varphi \left( \delta  \right): = \sup\left\{ { \left| {\xi  - \zeta } \right|:\left| {\omega \left( \xi  \right) - \omega \left( \zeta  \right)} \right| \leqslant \delta } \right\} $ is continuously defined on $ [0, \tau'] $ with some $ 0<\tau'\leqslant1/2 $. Moreover, it satisfies the conditions $ \varphi(0)=0 $ and $ \varphi(x)>0 $ for $ x \in  (0, \tau'] $. Additionally, $ \varphi $ has certain integrability properties, specifically,
\[ - \int_0^{{\tau'}} {\frac{{\varphi (x)}}{{x\ln x}}dx}  <  + \infty .\]
\end{enumerate}

In what follows, let us make some important remarks on \ref{aph2} and \ref{aph3}, which involve a considerable weakening of the frequency mapping properties. These remarks are crucial, as they elucidate some of the \textit{main motivations} of this paper from various perspectives. For the convenience of readers, let us provide a summarizing overview here: Existing works all require the frequency mapping to be bi-Lipschitz or to have extremely poor regularity (such as being nowhere differentiable). However, these new assumptions we introduce can perfectly resolve this difficulty, \textit{allowing the frequency mapping to have explicit arbitrarily weak regularity and extremely strong degeneracy.} Experienced readers may skip these comments and proceed directly to our main result, Theorem \ref{FPKAMT1}.
\begin{enumerate}[label=(C\arabic*), ref=(C\arabic*)]
\item \label{comc1}  In general, for perturbations $ \varepsilon P(y,x,\xi) $ that have Lipschitz continuity with respect to the parameter, it is always required that the frequency mapping $ \omega(\xi) $ is bidirectional Lipschitz continuous, e.g., see \cite{MR1370761,MR2837120,MR2842962,MR2819413}. This imposes stronger limitations on practical applications because such a frequency mapping $ \omega(\xi) $ is \textit{essentially locally linear}.	 As a result, the vast majority of existing work is unable to address situations beyond this particular case.

At the same time, the results given in \cite{DL,MR4806284,MR4731278,DTL} also cannot deal with the frequency mapping $ \omega(\xi) $ with \textit{good} regularity (or  \textit{degeneracy}, i.e., locally smoother than linearly), and sometimes $ \omega(\xi) $ has to be nowhere differentiable. As an example, for $ P(y,x,\xi):=\xi^3 $ in the case $ n=1 $, the introduced  modulus of continuity $ \varpi_1 (\sigma)$ of $ P $ corresponds to the Lipschitz modulus $ \sigma $, although it is smoother at $ 0 $. Then the (uniform) weak convex modulus $ \varpi_2(\sigma) $ of the frequency mapping $ \omega(\xi) $ must be at least $\sigma $, which yields almost surely linearity as we mentioned before. Unfortunately, if the perturbation has a weaker regularity with respect to $ \xi $, e.g., the H\"older type $ P(y,x,\xi):=|\xi|^{1/2} $, then the frequency mapping would be at least  nowhere differentiable. This is precisely why we previously emphasized that the limitations of uniformity make constructing a regularity case that is \textit{strictly weaker} than the  Lipschitz type particularly challenging. All the above can be explained technically \cite{MR4731278}, but it seems somewhat odd philosophically: if we can handle the bad cases, why cannot we handle the good ones? Addressing this question is also one of the main motivations of this paper.

	Fortunately, our new assumptions \ref{aph2} and \ref{aph3} will overcome these essential difficulties, allowing the frequency mapping to admit \textit{degeneracy}. Note that the translation of a mapping does not change its properties, so below we discuss only the asymptotic properties near $ 0 $.

	\begin{enumerate}[label=(\roman*)]
		\item  Consider the  case of $ n=1 $ with $ P(y,x,\xi):=\xi^3 $. Then $ \omega(\xi):=\xi^3 $ satisfies \ref{aph2} and \ref{aph3}, and one observes that the inverse $ {\omega ^{ - 1}}\left( \xi  \right) = {\xi ^{1/3}} $ is not Lipschitz continuous, but instead admits H\"older continuity with optimal H\"older exponent $ 1/3 $. Therefore it is not a bidirectional Lipschitzian mapping.
		\item  Moreover, it is easy to see that \ref{aph2} \textit{never} requires the frequency mapping to be too bad (e.g., nowhere differentiable), as previous works \cite{DL,MR4806284,MR4731278,DTL} \textit{must require} in many cases.  It only requires that the mapping admits a relative singularity to the perturbation.  As illustrations, in the case of $ n=1 $ (the higher-dimensional case can be discussed in a similar manner), $ P(y,x,\xi):={\rm{sign}(\xi)}|\xi|^{1/2} $ with $ \omega(\xi):=\xi^{1/3} $; and even $ P\left( {y,x,\xi } \right): = {\left| \xi  \right|^\alpha } $ with $ \omega \left( \xi  \right): = {\rm{sign}(\xi)}{\left( { - \ln |\xi| } \right)^{ -\beta  }} $ (here $ \omega(0):=0 $) for all $\alpha,\beta>0 $. Leaving aside the singularities at $ 0 $, both of these frequency mappings are $ C^\infty $ (note that they are only defined on a subset  $ [0,\tau']\subset[0,1/2] $).
	\end{enumerate}

\item \label{comc2} For example, the linear mapping $\omega(\xi) = a\xi + b \in \mathbb{R}^n$ admits $\varphi(\delta) = \mathcal{O}^{\#}(\delta)$ as $\delta \to 0^+$, whenever no component of $a$ is $0$ (we call this a non-degenerate linear mapping). However, if $a$ has one component equal to $0$ (the degenerate case), then such a continuous function $\varphi(\delta)$ satisfying $\varphi(0) = 0$ in \ref{aph3} does not exist. For instance, consider $\omega(\xi) = (0, \xi_2)$ in the case of $n= 2$. Here, $|\omega(\xi) - \omega(\zeta)| = |\xi_2 - \zeta_2| \leqslant \delta \ll 1$ does not imply that $|\xi - \zeta|$ is small, specifically $|\xi_1 - \zeta_1|$ could be large. Indeed, \ref{aph3} implies that $ \omega (\xi) $ is injective; hence, $\varphi(\delta)$ represents  the modulus of continuity of $\omega^{-1}(\xi)$, i.e., $\varphi \left( \delta  \right) = {\sup _{\left| {\xi  - \zeta } \right| \leqslant \delta }}\left| {{\omega ^{ - 1}}\left( \xi  \right) - {\omega ^{ - 1}}\left( \zeta  \right)} \right |$.
	
\item \label{comc3} To illustrate the \textit{degeneracy} that we allow in the frequency mapping, we provide an explicit example here. In view of the integrability in \ref{aph3}, a critical case using a mollifier in $ \mathbb{R}^1$ is 
\begin{equation}\notag
	\omega \left( \xi  \right) = \bar \omega  + \text{sign}\left( \xi  \right)\exp \left( { - {{\left| \xi  \right|}^{ - \alpha }}} \right)
\end{equation}
with $ \xi  \in \left[ { - 1,1} \right] $ and $ \alpha  > 0 $, where $ \bar \omega $ is Diophantine. Clearly, this function is  very smooth at $ 0 $, i.e.,  all derivatives vanish at $ 0 $. Consequently, $ \omega(\xi) $ can already be nearly horizontal at $ 0 $. See Figure \ref{fig:sub1} for its shape. 
 \begin{figure}[htbp]
	\centering
	
	\begin{subfigure}[b]{0.47\textwidth}
		\centering
		\includegraphics[width=\textwidth]{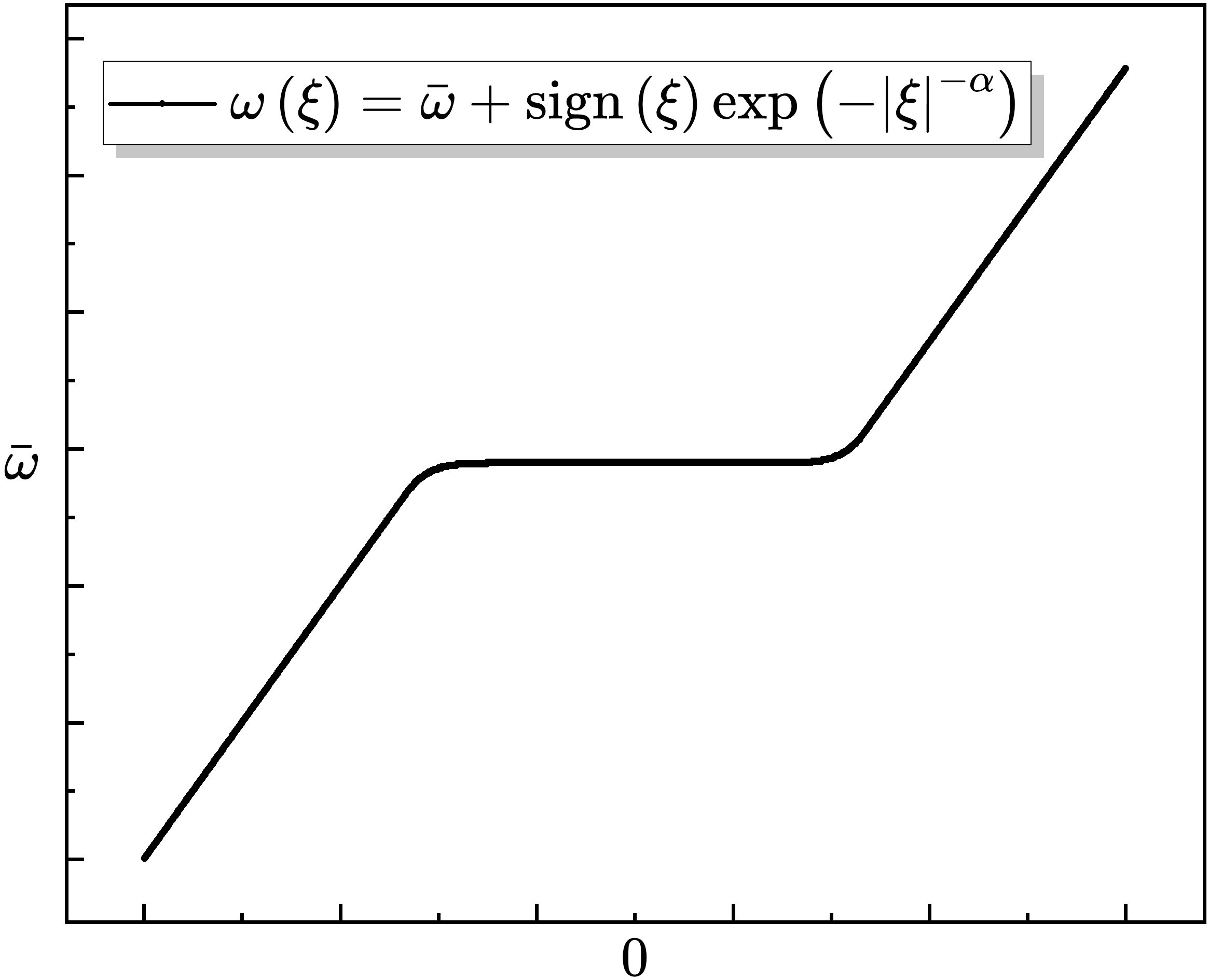}
		\caption{The shape of the frequency mapping $ \omega(\xi) $.}
		\label{fig:sub1}
	\end{subfigure}
	\hfill 
	\begin{subfigure}[b]{0.444\textwidth}
		\centering
		\includegraphics[width=\textwidth]{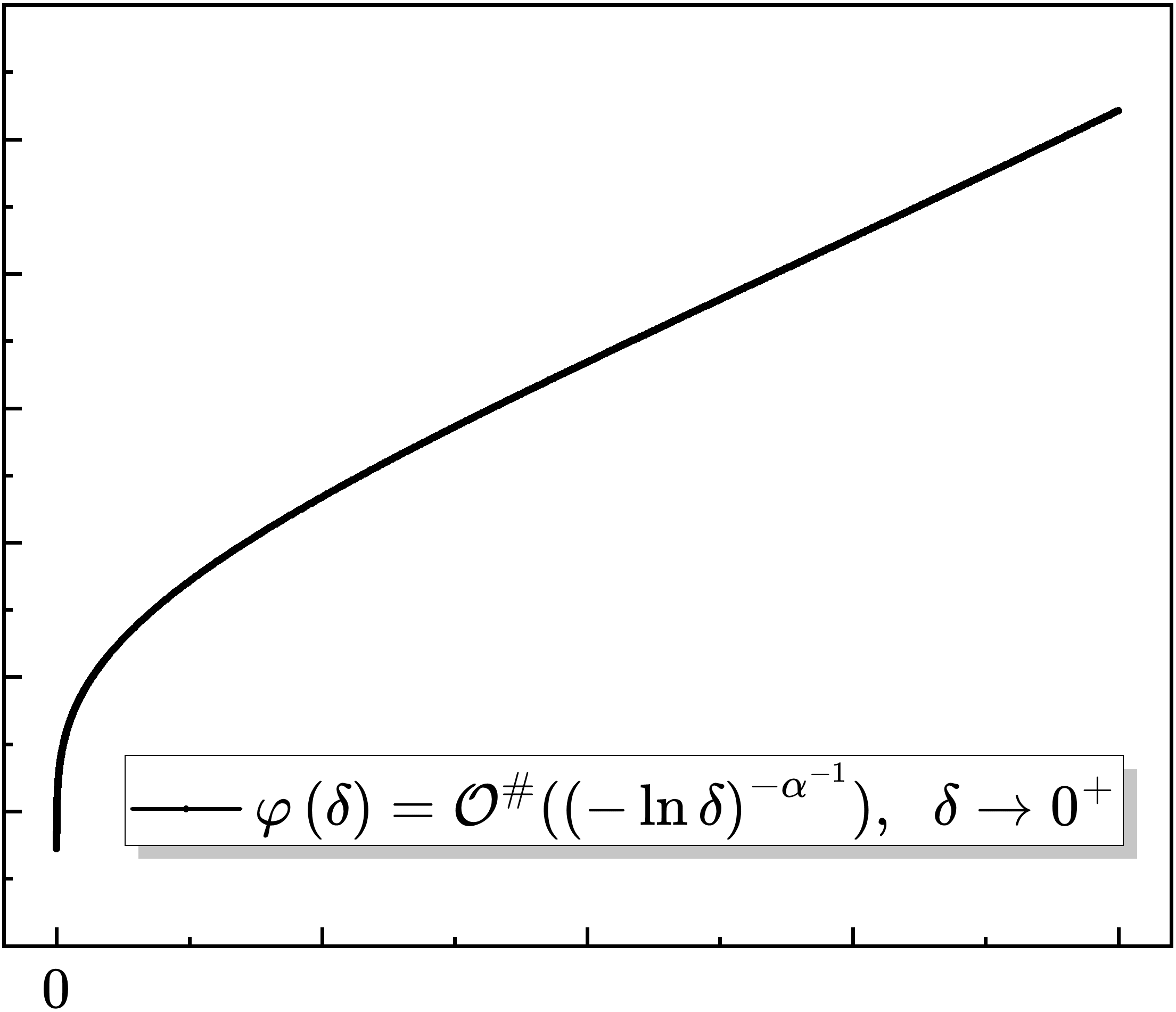}
		\caption{The shape of the function $\varphi(\delta) $.}
		\label{fig:sub2}
	\end{subfigure}
	\caption{A degenerate and  critical case in \ref{aph3} using a  mollifier}
	%\label{FIGWBA5}
\end{figure}
 This is one of the degenerate cases we referred to in Comment \ref{comc1}. One can verify that $ \varphi \left( \delta  \right) =\mathcal{O}^\#({\left( { - \ln \delta } \right)^{ - {\alpha ^{ - 1}}}})  $ as $ \delta \to0^+ $, see Figure \ref{fig:sub2} for its shape. Thus, \ref{aph3} is satisfied due to
\begin{equation}\label{FPKAMXING}
	- \int_0^{{\tau'}} {\frac{{\varphi (x)}}{{x\ln x}}dx}  =\mathcal{O}^\#\left(\int_0^{{\tau'}} {\frac{1}{{x{{\left( { - \ln x} \right)}^{1 + {\alpha ^{ - 1}}}}}}dx}\right)   =\mathcal{O}^\#\left({\int_1^{ + \infty } {\frac{1}{{{t^{1 + {\alpha ^{ - 1}}}}}}dt} }\right)   <  + \infty ,
\end{equation}
whenever $ 0<\tau' \leqslant 1/2 $. The criticality is mainly reflected in the convergence rate of the integrand in \eqref{FPKAMXING}. In fact, one can construct many more critical cases using the same approach.

\item \label{comc4} As previously shown, \ref{aph3} is mainly proposed to accommodate situations where the frequency mapping is too smooth. When perturbations under consideration have weak regularity with respect to the parameter, such as at most Lipschitz continuous or even just H\"older continuous, then \ref{aph3} automatically holds according to \ref{aph2}. In other words, the corresponding KAM theorem \textit{does not} require \ref{aph3} in this case (see \cite{MR4731278,DL,DTL,MR4806284} for similar cases, although the conditions required there are much stronger than those in this paper), e.g., Corollary \ref{FPKAMCORO1} concerning H\"older regularity (or Theorem \ref{th23} in the infinite-dimensional case).

\item \label{comc5}
As previously mentioned, one of the  central contributions of the present paper is that the required continuity of the Hamiltonian system with respect to the parameter $ \xi $ can be \textit{explicitly} prescribed and \textit{arbitrarily weak} (see also Remark \ref{FPKAMRE1.2}).	To illustrate this we restrict attention to the modulus of continuity $ \varpi $ of the perturbation $ P $ in $ \xi $.	Commonly used modulus of continuity are
\begin{itemize}
	\item Lipschitz: $\varpi_1(x)\sim x$, $ x \to 0^+ $
	\item $\alpha$-H\"older $(0<\alpha<1)$: $\varpi_2(x)\sim x^\alpha$, $ x \to 0^+ $
	\item Logarithmic H\"older: $\varpi_3(x)\sim(-\ln x)^{-1}$, $ x \to 0^+ $
	\item iterated Logarithmic H\"older: $\varpi_4(x)\sim(\ln(-\ln x))^{-1}$, $ x \to 0^+ $
\end{itemize}
and still \textit{weaker} ones are allowed.  Here, for any two modulus of continuity $ \varpi_A $ and $ \varpi_B $, we say that $ \varpi_A $ is weaker than $ \varpi_B $, if $ \mathop {\overline {\lim } }\nolimits_{x \to 0^+ } {\varpi _A (x)}/{\varpi _B (x)} <  + \infty  $. See \cite{MR4731278,CCM,TXL} for further explicit illustrations.		Geometrically, Lipschitz regularity corresponds to ``rectangular'' shapes, whereas H\"older behaves like ``almost circular'' neighbourhoods--such a description, however, proves far too restrictive in applications.	The last two modulus of continuity describe much steeper, narrower cones near the origin.

\begin{figure}[htbp]
	\centering

	\includegraphics[width=0.5\textwidth]{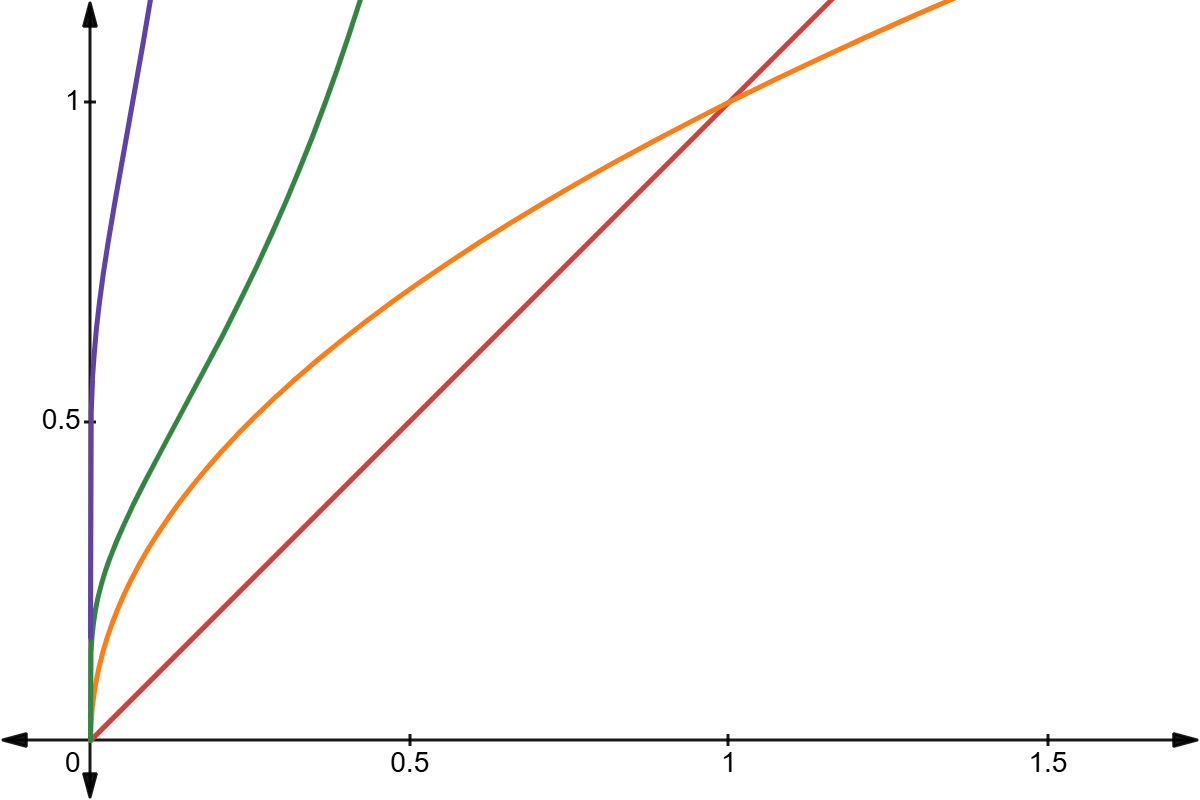}
	\caption{Comparison of the four modulus of continuity near zero: 
		From top to bottom, the curves show $ \varpi_4 $ (purple), $ \varpi_3 $ (green),  $ \varpi_2 $ (orange), and $ \varpi_1 $ (red). Their steepness decreases successively.}
	\label{FIGMOC}
\end{figure}
Figure \ref{FIGMOC} compares the four modulus of continuity in a neighbourhood of $ x=0 $, making the hierarchy transparent.	Since modulus of continuity of arbitrarily weak strength are permitted, our KAM theorems in this paper apply to an extremely broad class of  systems.
\end{enumerate}

Now, our main KAM results in this paper, namely frequency-preserving KAM in the parameterized settings, are stated as follows.

\begin{theorem}[Frequency-preserving KAM]\label{FPKAMT1}
Assume \ref{aph1}, \ref{aph2} and \ref{aph3}. Then there exists a sufficiently small $ {\varepsilon _0} > 0 $ such that, for any $ 0 < \varepsilon  < {\varepsilon _0} $,  one can find some $ {\xi^* }  $ near $ \xi_0 $ such that the	perturbed Hamiltonian system $ H\left( {y,x,{\xi^* },\varepsilon } \right) $ with parameter $ \xi^* $ in \eqref{FPKAMHamilton} admits an analytic quasi-periodic invariant torus with toral frequency $ \Upsilon=\omega \left( {{\xi _0}} \right) $. In particular, $  {\xi ^ * } = {\xi _0} +\mathcal{O}\left(\varphi \left( \varepsilon  \right) { - \int_0^{\varepsilon } {\frac{{\varphi (x)}}{{x\ln x}}dx} } \right) $ as $ \varepsilon \to 0^+ $.

\setcounter{footnote}{0}
\renewcommand{\thefootnote}{\fnsymbol{footnote}}
 Moreover,  there exists a continuous family of  $\xi^*(\hat\xi_0)$ with respect to $\hat\xi_0\in \mathscr{V}$, where $ \vartheta:=\min_{\xi\in\partial \mathcal{V}}\left\{|\omega(\xi)-\omega(\xi_0)|\right\} $, and 
\[\mathscr{V}=\left\{\hat\xi_0\in \mathcal{V}:|\omega(\hat \xi_0)-\omega(\xi_0)|<\vartheta,\;|\langle k,\omega(\hat\xi_0)\rangle|>{\gamma} |k|^{-\tau},\;\forall 0 \ne k \in {\mathbb{Z}^n}\right\},\]
such that the perturbed Hamiltonian system $H(y,x,\xi^*(\hat\xi_0),\varepsilon)$ admits an analytic quasi-periodic invariant torus with toral frequency $\omega(\hat\xi_0)$, and $\xi^*(\hat\xi_0)\rightarrow\hat\xi_0$ as $\varepsilon\to 0^+$. In particular, the continuity of $\xi^*(\hat\xi_0)$ is governed by a modulus of continuity $\varpi$, where $ \varpi $ is the weaker of the modulus of continuity of $ \omega(\xi) $ and its inverse $ \omega^{-1}(\xi) $.
\end{theorem}
\begin{remark}
We emphasize that the assumptions on the frequency mapping and the perturbation introduced in~\cite{DSL25,DL,DTL,MR4806284,MR4731278} should be replaced by those developed in the present paper.
\end{remark}
\begin{remark}
In fact, beyond the neighborhood $ \mathcal{V} $ of $ \xi_0 $ given in \ref{aph1}, the regularity of the frequency mapping and perturbations with respect to the parameter $ \xi $ can be arbitrarily weak, such as nowhere differentiability, nowhere H\"older continuity, or even nowhere continuity. While in $ \mathcal{V} $, they could be merely continuous.
\end{remark}
\begin{remark}
The KAM iteration belongs to the Picard scheme, which, unlike the Tonelli scheme, is incapable of extracting a convergent subsequence. \cite{DL} provided a counterexample   illustrating this, specifically demonstrating that the sequence can exhibit infinite oscillations. However, our parameter translation technique based on assumptions \ref{aph1}, \ref{aph2}, and \ref{aph3} can avoid this pitfall, resulting in a uniformly convergent sequence of frequency-preserving parameters, and their asymptotic behavior can be  estimated explicitly according to \ref{aph3}.
\end{remark}

Below, we provide some explicit applications of Theorem \ref{FPKAMT1}.

\begin{corollary}[H\"older continuity]\label{FPKAMCORO1}
 Assume, in addition to \ref{aph1} and \ref{aph2}, that the  perturbation $ P $ is $ \alpha $-H\"older near $ \xi_0 $ with some $ 0<\alpha<1 $, and is injective with respect to $ \xi $ near $ \xi_0 $ for fixed $ y,x $. Then there exists a sufficiently small $ {\varepsilon _0} > 0 $ such that, for any $ 0 < \varepsilon  < {\varepsilon _0} $,  one can find some $ {\xi^* }  $ near $ \xi_0 $ such that the	perturbed Hamiltonian system $ H\left( {y,x,{\xi^* },\varepsilon } \right) $ with parameter $ \xi^* $ in \eqref{FPKAMHamilton} admits an analytic quasi-periodic invariant torus with toral frequency $ \Upsilon=\omega \left( {{\xi _0}} \right) $. In particular, $ {\xi ^ * } = {\xi _0} +\mathcal{O}\left(\varphi \left( \varepsilon  \right) { - \int_0^{\varepsilon } {\frac{{\varphi (x)}}{{x\ln x}}dx} } \right) $ as $ \varepsilon \to 0^+ $.
\end{corollary}
\begin{remark}\label{FPKAMRE1.2}
	One could obtain more general modulus of continuity cases using the notations  introduced in \cite{MR0538680,DTL,MR4731278,MR4806284,CCM,TXL} (see also Comment \ref{comc5}). As a consequence, we can handle explicit cases where perturbations only admit continuity (beyond H\"older) with respect to the parameter, e.g., the (iterated) Logarithmic H\"older type
	\begin{align*}
	H\left( {y,x,\xi ,\varepsilon } \right) = &\left\langle {\omega \left( {{\xi _0}} \right),y} \right\rangle  + \sum\limits_{i = 1}^n {{\rm{sign}}\left( {{\xi _i}} \right){{\left( {\ln \left( { - \ln \left| {{\xi _i}} \right|} \right)} \right)}^{ - \beta }}{y_i}} \\
	&  + \varepsilon \sum\limits_{i = 1}^n {{{\left( { - \ln \left| {{\xi _i}} \right|} \right)}^{ - \gamma }}\left( {\sin {x_{n-i}} +(\sin {y_i})^2} \right)} ,
	\end{align*}
where $\beta,\gamma>0 $, and $ \xi \in \mathcal{O}:=[-1/4,1/4]^n $. Here we supplement the definitions at $ \xi_i=0 $ in terms of limits. 
\end{remark}

\begin{corollary}\label{FPKAMCORO2}
	Assume that the Hamiltonian system in \eqref{FPKAMHamilton} admits a non-degenerate linear frequency mapping. Then almost all (in a full Lebesgue measure sense) frequencies in  $ \omega(\mathcal{O}) $ can be preserved in KAM by selecting  appropriate parameters in $ \mathcal{O}^o $, whenever $ \varepsilon $ is sufficiently small and the perturbation $ P $ is Lipschitz continuous with respect to the parameter.
\end{corollary}
\begin{remark}
	By utilizing the Mean Value Theorem,  this corollary can  be generalized to the case where  $ \det \left( {D\omega \left( \xi  \right)} \right) \ne 0 $ on $  \mathcal{O} $.
\end{remark}

The remainder of the paper is organized as follows. Section \ref{FPKAMCONTEREXAMPLES} provides a number of counterexamples that help us understand the frequency-preserving KAM in detail. Among those, we show that the Internal Condition \ref{aph1}, the Relative Singularity Condition \ref{aph2} and the Controllability Condition \ref{aph3} are indispensable in the sense of frequency-preserving. We also introduce a new idea (Theorem \ref{FPKAMT2}) to adjust the Relative Singularity Condition \ref{aph2} by utilizing extra parameters. This approach is somewhat similar to bifurcation theory, from a philosophical standpoint. In addition, we demonstrate the further applicability of our results to the partial frequency-preserving KAM theory (Theorem \ref{PFPKAM}) and the infinite-dimensional KAM theory \textit{without any spectral asymptotics} (Theorem \ref{th23}), among others. These contributions are also entirely new to this paper. In conclusion, the approach of this paper, based on relative singularity and controllability, is indeed a \textit{generic} one, allowing us to apply it to an extremely broad range of problems, especially \textit{explicit} systems with \textit{arbitrarily weak regularity}\footnote{To our knowledge, this is the first rigorous approach to achieve this.} and \textit{degeneracy}. Finally, the proofs of Theorem \ref{FPKAMT1} and Corollaries \ref{FPKAMCORO1} and \ref{FPKAMCORO2} are given in Sections \ref{FPKAMsec-4} and \ref{FPKAMsec-5}. These proofs follow a modified parameter translation technique and a modified quasi-linear (quasi-Newtonian) type KAM iteration. They also strongly rely on the Relative Singularity Condition \ref{aph2} and the Controllability Condition \ref{aph3}, which we initially introduced.
\section{Counterexamples, parallel applicability and further comments}\label{FPKAMCONTEREXAMPLES}
This section is mainly divided into two parts. The first part includes Sections \ref{FPKAMSEC2.1} and \ref{FPKAMSEC2.2}, and provides several counterexamples based on different perturbations, aiming to illustrate the frequency-preserving KAM almost completely. In particular, we show that our assumptions, namely the Internal Condition \ref{aph1}, the Relative Singularity Condition \ref{aph2} and the Controllability Condition \ref{aph3}, are indeed indispensable in the sense of frequency-preserving. The second part involves Sections \ref{FPKAMLESS}, \ref{FPKAMPAR} and \ref{FPKAMINFINITE}, and extends our KAM results to the partial frequency-preserving KAM, the infinite-dimensional KAM and etc.

\subsection{Counterexamples via perturbations in Theorem \ref{FPKAMT1}}\label{FPKAMSEC2.1}
Consider a specific analytic  Hamiltonian system \eqref{FPKAMHamilton} in the case $ n=2 $:  
\begin{equation}\notag 
	H\left( {y,\xi,\varepsilon } \right) = \left\langle {\omega \left( \xi  \right),y} \right\rangle  + \varepsilon P\left( {y,\xi  } \right),
\end{equation}
where $ 0<\varepsilon\ll 1$, and the frequency mapping $ \omega \left( \xi  \right) = \left( {{{\bar \omega }_1} + \tilde \omega_1 \left( {{\xi }} \right),{{\bar \omega }_2} + \tilde \omega_2 \left( {{\xi }} \right)} \right) $ with $ \left( {{{\bar \omega }_1},{{\bar \omega }_2}} \right) $ being Diophantine. Note that according to the symplectic structure, $ \dot x (t,\xi,\varepsilon) ={H_y}\left( {y,\xi ,\varepsilon} \right) $ and $ \dot y (t,\xi,\varepsilon) =-{H_x}\left( {y,\xi ,\varepsilon} \right) =0$.
Then the unperturbed torus for fixed $ y=y_0 $ can be written as $ x(t,\xi,0) = \omega \left( \xi  \right)t + x_0 $, associated with initial value $ y_0$ and $x_0 $, and the perturbed one for fixed $ y=y_0 $ is $ x(t,\xi,\varepsilon) = \left( {\omega \left( \xi  \right) + \varepsilon {\partial _y}P\left( {y,\xi } \right)} \right)t + x_0 $.  Therefore, to obtain a $ \left( {{{\bar \omega }_1},{{\bar \omega }_2}} \right) $-frequency-preserving KAM torus with parameter $ {\xi ^ * } = \left( {\xi _1^ * ,\xi _2^ * } \right) \in \mathcal{O} $ (the parameter set $ \mathcal{O} \subset \mathbb{R}^2 $ will be specified later), the following must be satisfied:
\begin{equation}\label{FPKAMlianli}
	\left\{ \begin{gathered}
		\tilde \omega_1 \left( {\xi ^ * } \right) + \varepsilon {\partial _{y_1}}P\left( {y,{\xi ^ * }} \right) = 0, \hfill \\
		\tilde \omega_2 \left( {\xi ^ * } \right) + \varepsilon {{\partial _{y_2}}P}\left( {y,{\xi ^ * }} \right) = 0. \hfill \\ 
	\end{gathered}  \right.
\end{equation}

\subsubsection{Non-uniqueness of the parameter corresponding to frequency-preserving}
The parameter corresponding to frequency-preserving KAM torus might not be unique, in other words, there could be many KAM tori with the prescribed frequency.

Let $ \mathcal{O}:=[-1,1]\times[0,3\pi] $, and set $ {{\tilde \omega }_1}\left( \xi  \right) = {\xi _1} $, $ {{\tilde \omega }_2}\left( \xi  \right) = \sin {\xi _2} $, and $ P\left( {y,\xi } \right) =  - {y_2} $. Note that $ \xi_0 $ is not necessarily unique, i.e., $ {\xi _0} = (0,\pi) $ or $(0,2\pi)  $. Near these two points, \ref{aph1} to \ref{aph3} automatically hold (with $ \varphi(\delta)=\mathcal{O}^\#(\delta) $). This is because \ref{aph2} is a local assumption, it does not require the frequency mapping $ \omega(\xi) $ to be a global injection. Now, for $0< \varepsilon\ll 1 $, the frequency-preserving equations in  \eqref{FPKAMlianli} give 
\[\left\{ \begin{gathered}
	\xi _1^ *  = 0, \hfill \\
	\sin {\xi _2^*} - \varepsilon  = 0, \hfill \\ 
\end{gathered}  \right. \Rightarrow \left\{ \begin{gathered}
	\xi _1^ *  = 0, \hfill \\
	\xi _2^ *  = \arcsin \varepsilon ,\pi  - \arcsin \varepsilon ,2\pi  + \arcsin \varepsilon ,3\pi  - \arcsin \varepsilon , \hfill \\ 
\end{gathered}  \right.\]
which shows the non-uniqueness of the parameter corresponding to frequency-preserving. It should be mentioned that $ \xi _2^ *  = \arcsin \varepsilon  $ and $ \xi _2^ *  = 3\pi-\arcsin \varepsilon $ are not derived from our KAM theorem (Theorem \ref{FPKAMT1}) because they are both at a positive distance (independent of $ \varepsilon>0 $) from their respective $ \xi_0 $, i.e.,
\[\mathop {\lim }\limits_{\varepsilon  \to {0^ + }} \left| {\left( {0,\arcsin \varepsilon } \right) - \left( {0,\pi } \right)} \right| = \mathop {\lim }\limits_{\varepsilon  \to {0^ + }} \left| {\left( {0,3\pi  - \arcsin \varepsilon } \right) - \left( {0,2\pi } \right)} \right| = \pi  \ne 0.\]
 But $ \xi _2^ *  = \pi-\arcsin \varepsilon  $ and $ \xi _2^ *  = 2\pi+\arcsin \varepsilon  $ are compatible with our KAM theorem  by verifying \ref{aph1} to \ref{aph3}, respectively.

The previous linear counterexample concerns the case that $ \xi_0 $ could be selected differently. Next, we give a nonlinear Hamiltonian system that admits infinitely (uncountable) many frequency-preserving KAM tori, and $ \xi_0 $ must be unique at this time. Let $ \mathcal{O}:=[-1,1]^2 $, and set $ {{\tilde \omega }_1}\left( \xi  \right) = {\xi _1} $, $ {{\tilde \omega }_2}\left( \xi  \right) = {\xi _2} $, and $ P\left( {y,\xi } \right) =  - \left( {y_1^2/2 + {y_2} + y_2^2/2} \right) $. Then \ref{aph1} to \ref{aph3} automatically hold (with $ \varphi(\delta)=\mathcal{O}^\#(\delta) $) near $ \xi_0 =(0,0)$, and the equations in \eqref{FPKAMlianli} turn to
\[\left\{ \begin{gathered}
	\xi _1^ *-\varepsilon y_1  = 0, \hfill \\
	{\xi _2^*} - \varepsilon (1+y_2) = 0, \hfill \\ 
\end{gathered}  \right. \Rightarrow \left\{ \begin{gathered}
	\xi _1^ *  =\varepsilon y_1, \hfill \\
	\xi _2^ *  = \varepsilon (1+y_2) . \hfill \\ 
\end{gathered}  \right.\]
If $ y \in G:=[-1,1]^2 $, then $ \xi^* $ could be infinitely many, i.e., $ {\xi ^ * } = \left( {\varepsilon {y_1},\varepsilon \left( {1 + {y_2}} \right)} \right) $ with any initial value $ y = \left( {{y_1},{y_2}} \right) \in G $ fixed (recall that $ \dot y =  - {H_x}\left( {y,\xi ,\varepsilon } \right) = 0 $). In other words, on any level set near the origin, we can adjust the parameter such that the prescribed frequency is preserved. Interestingly, if this counterexample takes no parameters, i.e., $ {{\tilde \omega }_1}\left( \xi  \right)= {{\tilde \omega }_2}\left( \xi  \right) = 0 $, then the $ (\bar \omega_1, \bar \omega_2) $-frequency-preserving KAM  torus may not exist, as long as we choose $ G:=[-1/2,1/2]^2 $ (because $ y_2=-1 \notin [-1/2,1/2] $). If $ G:=[-1,1]^2 $, then a unique $ (\bar \omega_1, \bar \omega_2) $-frequency-preserving KAM  torus does exist on $ \mathbb{T}^n \times \{(0,-1)\} $, which is completely different from the parameterized case.

\subsubsection{Non-differentiability of $ \xi^* $ with respect to $ \varepsilon $}
The obtained parameter $ \xi^* $ might be non-differentiable with respect to $ \varepsilon $ at $ 0 $.

Let $ \mathcal{O}:=[-1,1]^2 $, and set $ {{\tilde \omega }_1}\left( \xi  \right) = {\xi _1} $, $ {{\tilde \omega }_2}\left( \xi  \right) = \xi _2^3 $, and $ P\left( {y,\xi } \right) =  - {y_2} $. Then \ref{aph1} to \ref{aph3} automatically hold. To be more precise, note that $ \left| {\omega \left( \xi  \right) - \omega \left( \zeta  \right)} \right| \leqslant \delta  $ for $ 0<\delta\ll1 $ means $ \left| {\left( {{\xi _1} - {\zeta _1},\xi _2^3 - \zeta _2^3} \right)} \right| \leqslant \delta  $, then $ \left| {{\xi _1} - {\zeta _1}} \right| \leqslant \delta /2 $, and $ \left| {\xi _2^3 - \zeta _2^3} \right| \leqslant \delta /2 $. Therefore, one concludes that $ \left| {{\xi _2} - {\zeta _2}} \right| = {\mathcal{O}^\# }\left( {{\delta ^{1/3}}} \right) $. Taking $ \varphi \left( \delta  \right) = \max \left\{ {{\mathcal{O}^\# }\left( \delta  \right),{\mathcal{O}^\# }\left( {{\delta ^{1/3}}} \right)} \right\} = {\mathcal{O}^\# }\left( {{\delta ^{1/3}}} \right) $, it is easy to verify that for any $ 0<\tau'\leqslant1/2 $, it holds
\[ - \int_0^{\tau'}  {\frac{{\varphi (x)}}{{x\ln x}}dx}  \lesssim  - \int_0^{\tau'}  {\frac{1}{{{x^{2/3}}\ln x}}dx}  \lesssim \int_0^{\tau'}  {\frac{1}{{{x^{2/3}}}}dx}  <  + \infty. \]
However, aiming to achieve $ \left( {{{\bar \omega }_1},{{\bar \omega }_2}} \right) $-frequency-preserving, the equations in \eqref{FPKAMlianli} turn to 
\[\left\{ \begin{gathered}
	\xi _1^ *  = 0, \hfill \\
	{(\xi _2^ * )^3} - \varepsilon  = 0, \hfill \\ 
\end{gathered}  \right. \Rightarrow \left\{ \begin{gathered}
	\xi _1^ *  = 0, \hfill \\
	\xi _2^ *  = {\varepsilon ^{1/3}} \hfill, \\ 
\end{gathered}  \right.\]
which implies that $ \xi^* $ is non-differentiable with respect to $ \varepsilon $ at $ 0 $.

\subsubsection{Nonuniformity of the  distance from $ \xi^* $ to $ \xi_0 $ with respect to $ \varepsilon $}
As shown in Theorem \ref{FPKAMT1}, the selected parameter $ \xi^* $ corresponding to frequency-preserving converges to $ \xi_0 $ whenever $ \varepsilon\to 0^+ $, and the convergence rate could be dominated by $ \mathcal{O}\left(\varphi(\varepsilon) { - \int_0^{\varepsilon} {\frac{{\varphi (x)}}{{x\ln x}}dx} } \right) $ with $ \varphi $ defined in \ref{aph3}. But this is an upper bound estimate and does not represent the real rate. It is worth mentioning that such a rate is not uniform with respect to $ \varepsilon $, it can be logarithmically slow, or arbitrarily fast (for any rate given in advance) via some particular frequency mappings.

Recall the critical case in Comment \ref{comc3}, i.e., 
\begin{equation}\label{FPKAMcri}
	\omega \left( \xi  \right) = \bar \omega  + \text{sign}\left( \xi  \right)\exp \left( { - {{\left| \xi  \right|}^{ - \alpha }}} \right)
\end{equation}
with $ \xi  \in \left[ { - 1,1} \right] $ and $ \alpha  > 0 $, where $ \bar \omega $ is Diophantine.
 Now, let us consider a simple parameterized Hamiltonian system with $ n=1 $:
\begin{equation}\label{FPKAMcri2}
	H\left( {y,\xi,\varepsilon } \right) =  {\omega \left( \xi  \right)y}  - \varepsilon y.
\end{equation}
Then \ref{aph1} to \ref{aph3} automatically hold, and it is easy to verify that $ \xi^* =(\omega-\bar\omega)^{-1}(\varepsilon)=(-\ln \varepsilon)^{-\lambda}$ with $ \lambda=\alpha^{-1} >0 $, which provides the logarithmically slow rate. One can also construct a slower convergence rate by increasing the regularity of $ \omega(\xi) $.  As for any given fast rate $ \rho(\varepsilon) $, where the monotonic odd function $ 0<\rho(x)=\mathcal{O}(x) $, then modifying $ \omega(\xi) $ in \eqref{FPKAMcri} to $ \omega \left( \xi  \right) = \bar \omega  + \rho^{-1}(\xi) $ one deduces from \eqref{FPKAMcri2} that $ \xi^*-\xi_0 =\rho(\varepsilon)$ (note that \ref{aph1} to \ref{aph3} still hold, because $ \rho^{-1}(\xi) $ admits weak regularity at $ 0 $, see Comment \ref{comc4}). It is indeed an interesting fact that the weak regularity of the frequency mapping actually makes the convergence rate faster.

\subsubsection{Nonexistence of the  frequency-preserving KAM torus when $ \mathcal{O} $ is disconnected}
We have previously assumed  that the parameter set $ \mathcal{O} $ is connected. However, if the  connectivity is removed, then the  frequency-preserving KAM torus may not exist, even  $ \mathcal{O} $ has relatively full Lebesgue measure. 

Let $ \mathcal{O}:=\left[ { - 1,1} \right] \times \left( {\left[ { - 1,1} \right]\backslash \left( {\left[ { - 1,1} \right] \cap \mathbb{Q}} \right)} \right) $, and set $ {{\tilde \omega }_1}\left( \xi  \right) = \xi _1 $, $ {{\tilde \omega }_2}\left( \xi  \right) = \xi _2 $, and $ P\left( {y,\xi } \right) =  - ({\xi _2}+1){y_2} $. Then $ \mathcal{O} $ admits full Lebesgue measure relative to $ [-1,1]^2 $, and \ref{aph1} to \ref{aph3} automatically hold (with $ \varphi(\delta)=\mathcal{O}^\#(\delta) $). The equations in \eqref{FPKAMlianli} become
\[\left\{ \begin{gathered}
	\xi _1^ * = 0, \hfill \\
	\xi _2^ *  - \varepsilon ({\xi _2^*}+1) = 0, \hfill \\ 
\end{gathered}  \right. \Rightarrow \left\{ \begin{gathered}
	\xi _1^ *  = 0, \hfill \\
	\xi _2^ *  = \varepsilon (1-\varepsilon)^{-1}. \hfill \\ 
\end{gathered}  \right.\]
Note that $ \varepsilon (1-\varepsilon)^{-1} \in \mathbb{Q} $ whenever $ \varepsilon\in \mathbb{Q} $ and $0< \varepsilon \ll 1 $, and hence, such a parameter $ \xi^* \notin \mathcal{O} $.

\subsubsection{Nonexistence of the  frequency-preserving KAM torus when $ \omega(\xi) $ is discontinuous}\label{FPKAMSUB213}
 Note that all of our assumptions \ref{aph1} to \ref{aph3} are based on the premise that the  frequency mapping and  perturbations are continuous with respect to the parameter. Below we will show that if the frequency mapping is not continuous (which of course also violates the assumptions, but does not conflict with some subsequent counterexamples, which require continuity of the frequency mapping), then the frequency-preserving KAM torus may not exist.

Let $ \mathcal{O}:=[-1,1]^2 $, and set $ {{\tilde \omega }_1}\left( \xi  \right) = \mathcal{M}\left( {{\xi _1}} \right) $, $ {{\tilde \omega }_2}\left( \xi  \right) = \xi _2 $, and $ P\left( {y,\xi } \right) =  - {\xi _2}{y_1} - {y_2} $, where $ \mathcal{M}\left( 0 \right) = 0 $, and $ \mathcal{M}\left( {{\xi _1}} \right) = \xi _1^{ - 1} $ for $ \xi_1 \ne 0 $. Then the frequency mapping is discontinuous with respect to $ 
\xi $ at $ (0,0) \in \mathcal{O}^o $, and the equations in \eqref{FPKAMlianli} read
\[\left\{ \begin{gathered}
	\mathcal{M}(\xi _1^ * ) - \varepsilon \xi _2^ *  = 0, \hfill \\
	\xi _2^ *  - \varepsilon  = 0, \hfill \\ 
\end{gathered}  \right. \Rightarrow \left\{ \begin{gathered}
	\xi _1^ *  = {\varepsilon ^{ - 2}}, \hfill \\
	\xi _2^ *  = \varepsilon,  \hfill \\ 
\end{gathered}  \right.\]
which implies that $ \xi_1^* \notin [-1,1]$ as long as $ \varepsilon>0 $ is sufficiently small, and therefore such a parameter solution $ \xi^* $ does not exist.

\subsubsection{Nonexistence of the  frequency-preserving KAM torus when $ P(y,x, \xi) $ is discontinuous with respect to $ \xi $}
Similar to the exposition in Section \ref{FPKAMSUB213}, we will show here that if perturbations are not continuous with respect to the parameter, then the frequency-preserving KAM torus may be destroyed.

Let $ \mathcal{O}:=[-1,1]^2 $, and set $ {{\tilde \omega }_1}\left( \xi  \right) = {\xi _1} $, $ {{\tilde \omega }_2}\left( \xi  \right) = \xi _2 $, but $ P\left( {y,\xi } \right) =  - \mathcal{N}\left( {{\xi _2}} \right){y_1} - {y_2} $, where $ \mathcal{N}\left( 0 \right) = 0 $, and $ \mathcal{N}\left( {{\xi _2}} \right) = \xi _2^{ - 2} $ for $ \xi_2 \ne 0 $. It is evident that the perturbation is discontinuous with respect to $ \xi $ at $ (0,0) \in \mathcal{O}^o$, and $ \varepsilon P(y,\xi)$ is small for fixed $ y,\xi $, whenever $ 0<\varepsilon \ll 1 $ (but the smallness is nonuniform). Then the $ \left( {{{\bar \omega }_1},{{\bar \omega }_2}} \right) $-frequency-preserving equations in  \eqref{FPKAMlianli} become
\[\left\{ \begin{gathered}
	\xi _1^ *  - \varepsilon \mathcal{N}(\xi _2^ * ) = 0, \hfill \\
	\xi _2^ *  - \varepsilon  = 0, \hfill \\ 
\end{gathered}  \right. \Rightarrow \left\{ \begin{gathered}
	\xi _1^ *  = {\varepsilon ^{ - 1}}, \hfill \\
	\xi _2^ *  = \varepsilon , \hfill \\ 
\end{gathered}  \right.\]
which implies that $ \xi_1^* \notin [-1,1]$ whenever $ 0<\varepsilon \ll 1 $, and therefore $ \xi^* $ does not exist. 

\subsubsection{Nonexistence of the  frequency-preserving KAM torus in the absence of the Internal Condition \ref{aph1}}
Here we show a crucial fact that the Internal Condition \ref{aph1} cannot be removed in the sense of frequency-preserving, by constructing the counterexample below.

Let $ \mathcal{O}:=[-1,1]\times[0,1] $, and set $ {{\tilde \omega }_1}\left( \xi  \right) = {\xi _1} $, $ {{\tilde \omega }_2}\left( \xi  \right) = {\xi _2} $, and $ P\left( {y,\xi } \right) = {y_2} $. Then \ref{aph1} fails because the unique parameter $ \xi_0=(0,0) \notin \mathcal{O}^o  $ ($ \xi_0 \in \mathcal{O} $), but \ref{aph2} and \ref{aph3} automatically hold (with $ \varphi(\delta)=\mathcal{O}^\#(\delta) $). Now, the equations in \eqref{FPKAMlianli} turn to 
\[\left\{ \begin{gathered}
	\xi _1^ *  = 0 ,\hfill \\
	\xi _2^ *  + \varepsilon  = 0, \hfill \\ 
\end{gathered}  \right. \Rightarrow \left\{ \begin{gathered}
	\xi _1^ *  = 0 ,\hfill \\
	\xi _2^ *  =  - \varepsilon  , \hfill \\ 
\end{gathered}  \right.\]
which implies that $ \xi_2^* \notin [0,1]$ due to $ \varepsilon>0 $, and therefore $ \xi^* $ does not exist.

Next, we construct another weaker counterexample, which contradicts both \ref{aph1} and \ref{aph3}. However, this is different from the above counterexample, which focuses on the internal property of the parameter $ \xi_0 $, here we focus on the internal property of the frequency value $ \omega(\xi_0)$.
%\begin{itemize}
%\item [Case 1:] The internal property of $ \xi_0 \in \mathcal{O}^o $ is indispensable, otherwise the frequency-preserving KAM torus may do not exist.

%Let $ \mathcal{O}:=[0,1]\times[-1,1] $, 

%\item [Case 2:]  The internal property of $ \Upsilon=\omega(\xi_0)\in (\omega(\mathcal{O}))^o $ is indispensable, otherwise the frequency-preserving KAM torus may do not exist. 

Let $ \mathcal{O}:=[-1,1]^2 $, and set $ {{\tilde \omega }_1}\left( \xi  \right) = {\xi _1} $, $ {{\tilde \omega }_2}\left( \xi  \right) = \xi _2^2 $, and $ P\left( {y,\xi } \right) = \xi_1 y_1+\left( {\xi _2^2 + 1} \right){y_2} $. Then \ref{aph1} fails because $\omega \left( \xi_0 \right)= \omega \left( {\left( {0,0} \right)} \right)=(\bar \omega_1, \bar \omega_2) $ is not an interior point of the set $ \omega \left( \mathcal{O} \right) =[\bar \omega_1 -1,\bar \omega_1 +1]\times \bar [\bar \omega_2,\bar \omega_2 +1]$ due to the non-negativity of $ \xi_2^2 $ at $ 0 $, but the internal property of $ \xi_0=(0,0) \in \mathcal{O}^o=(-1,1)^2 $ holds. As for \ref{aph2}, note that near $ \xi_0 $, we have
\begin{align*}
\mathop {\sup }\limits_{y,x \in G \times \mathbb{T}^n} \left| {P\left( {y,\xi } \right) - P\left( {y,\zeta } \right)} \right| &\leqslant \mathop {\sup }\limits_{y \in G} \left| {{\xi _1}{y_1} - {\zeta _1}{y_1}} \right| + \mathop {\sup }\limits_{y \in G} \left| {\xi _2^2{y_2} - \zeta _2^2{y_2}} \right|\\
& \leqslant {\text{diag}}G\left( {\left| {{\xi _1} - {\zeta _1}} \right| + \left| {\xi _2^2 - \zeta _2^2} \right|} \right)\\
& \leqslant 2{\text{diag}}G\left| {\left( {{\xi _1} - {\zeta _1},\xi _2^2 - \zeta _2^2} \right)} \right|\\
& = 2{\text{diag}}G\left| {\omega \left( \xi  \right) - \omega \left( \zeta  \right)} \right|,
\end{align*}
where $ {\text{diag}}G $ denotes    the diameter of region $ G $.
Therefore, by supplementing the definition ($ \omega $ and $ P $ agree on the properties of the parameter simultaneously), we prove that \ref{aph2} holds. However, \ref{aph3} fails due to the  non-monotonicity of $ \tilde \omega_2 $ at $ 0 $. To be more precise, let $ \xi=(0,a) $ and $ \zeta=(0,-a) $ with any $ a >0 $; then $ \left| {\omega \left( \xi  \right) - \omega \left( \zeta  \right)} \right| = 0 $, whereas $ \left| {\xi  - \zeta } \right| = 2a > 0 $. Consequently, the function $ \varphi(\delta) $ defined in \ref{aph3} fails to satisfy $ \varphi(0)= 0 $. Note that the equations in \eqref{FPKAMlianli} take the form
\[\left\{ \begin{gathered}
		\xi _1^ *  +\varepsilon\xi _1^ *=0, \hfill \\
		{(\xi _2^ * )^2} + \varepsilon ( {{{(\xi _2^ * )}^2} + 1} ) = 0, \hfill \\ 
	\end{gathered}  \right.\]
	then $ \xi_2^* $ does not exist for any $ \varepsilon>0 $ and neither does $ \xi^* $. This shows that the $ (\bar{\omega}_1,\bar{\omega}_2) $-frequency-preserving KAM torus cannot survive.

\subsubsection{Nonexistence of the  frequency-preserving KAM torus in the absence of the Relative Singularity Condition \ref{aph2} and the  Controllability Condition \ref{aph3}}\label{FPKAMSEC217}
Here we emphasize a significant point: If the Relative Singularity Condition \ref{aph2} and the Controllability Condition \ref{aph3} fail, then the frequency-preserving KAM torus may not exist. Noticing that \ref{aph2} and \ref{aph3} are not completely independent, we will construct a non-injective frequency mapping counterexample below.

 Let $ \mathcal{O}:=[-1,1]^2 $, and set $ {{\tilde \omega }_1}\left( \xi  \right) = {\xi _1} $, $ {{\tilde \omega }_2}\left( \xi  \right) = {\xi _1}{\xi _2} $, and $ P\left( {y,\xi } \right) =  - {y_2} $. It is evident that \ref{aph1} holds (with $ \xi_0=(0,0) $), but \ref{aph2} and \ref{aph3} fail, because for any $ j>1 $, taking $ a = \left( {0,1/j} \right) $ and $ b = \left( {0, - 1/j} \right) $ yields $ \left| {\omega \left( a \right) - \omega \left( b \right)} \right| \equiv 0 $, but $ \left| {a - b} \right| = \left| {\left( {0,2/j} \right)} \right| = 2/j > 0 $. Therefore, \ref{aph2} fails because the supremum is undefined; and  \ref{aph3} fails because  $ \varphi(\delta) \equiv \varphi(0)=\kappa>0 $ for all $ 0<\delta\leqslant 2 $, and the integrability condition also fails due to 
\[ - \int_0^{{{\tau'}}} {\frac{{\varphi (x)}}{{x\ln x}}dx}  = -\kappa \int_0^{{{\tau'}}} {\frac{1}{{x\ln x}}dx}   =  - \ln \left( { - \ln x} \right)|_0^{\tau'}  =  + \infty ,\]
provided with any $ 0<{\tau'}\leqslant1/2 $. Now, if the $ (\bar{\omega}_1,\bar{\omega}_2) $-frequency-preserving KAM torus exists, then the equations in \eqref{FPKAMlianli} become
\[\left\{ \begin{gathered}
	\xi _1^ *  = 0, \hfill \\
	\xi _1^ * \xi _2^ *  - \varepsilon  = 0, \hfill \\ 
\end{gathered}  \right. \Rightarrow \left\{ \begin{gathered}
	\xi _1^ *  = 0, \hfill \\
	\varepsilon  = 0, \hfill \\ 
\end{gathered}  \right.\]
which contradicts $ \varepsilon>0 $.

Another much more trivial counterexample could be considered as $ \omega(\xi)\equiv(\bar \omega_1,\bar \omega_2) $ (i.e., a constant mapping) and $ P(y,\xi) =-y_2$.

\subsection{Some strange phenomena via general small perturbations beyond Theorem \ref{FPKAMT1}}\label{FPKAMSEC2.2}
Note that the whole perturbation in Hamiltonian system \eqref{FPKAMHamilton} is in the form  $  \tilde P\left( {y,x,\xi, \varepsilon} \right) =\varepsilon P\left( {y,x,\xi  } \right) $. If it has some singularity with respect to $ \varepsilon $, then certain pathological behaviors may arise. Throughout this section, let us consider a specific analytic Hamiltonian system in the case $ n=2 $:
\[H\left( {y,\xi,\varepsilon } \right) = \left\langle {\omega \left( \xi  \right),y} \right\rangle  +  \tilde{P}\left( {y,\xi ,\varepsilon } \right),\]
where $ \xi=(\xi_1,\xi_2) \in \mathcal{O}:=[-1,1]^2 $, and  $ \omega \left( \xi  \right) = \left( {{{\bar \omega }_1} + \tilde \omega_1 \left( {{\xi_1 }} \right),{{\bar \omega }_2} + \tilde \omega_2 \left( {{\xi_2 }} \right)} \right) $ with $ \left( {{{\bar \omega }_1},{{\bar \omega }_2}} \right) $ being  Diophantine. The perturbation $ \tilde P\left( {y,\xi,\varepsilon } \right)  $ is still small whenever  $ 0<\varepsilon \ll 1$. Now, the $ \left( {{{\bar \omega }_1},{{\bar \omega }_2}} \right) $-frequency-preserving equations in \eqref{FPKAMlianli} with parameter $ {\xi ^ * } = \left( {\xi _1^ * ,\xi _2^ * } \right) \in \mathcal{O} $ become
\begin{equation}\label{FPKAMlianli2}
	\left\{ \begin{gathered}
		\tilde \omega_1 \left( {\xi_1 ^ * } \right) + {\partial _{y_1}}\tilde{P}\left( {y,{\xi ^ * },\varepsilon } \right) = 0, \hfill \\
		\tilde \omega_2 \left( {\xi_2 ^ * } \right) + {\partial _{y_2}}{\tilde P}\left( {y,{\xi ^ * },\varepsilon } \right) = 0, \hfill \\ 
	\end{gathered}  \right.
\end{equation}
because $ H(y, \xi, \varepsilon) $ is independent of the angle variable $ x $.

\subsubsection{Almost everywhere discontinuity of $ \xi^* $ with respect to $ \varepsilon $}
If the small perturbation $ \tilde{P} $ is discontinuous with respect to the  parameter $ \xi $, then the obtained parameter $ \xi^* $ corresponding to frequency-preserving KAM torus may be  discontinuous with respect to $ \varepsilon $.

 For example, let $ {{\tilde \omega }_1}\left( {\xi _1 } \right) = \xi _1  $, $ {{\tilde \omega }_2}\left( {\xi _2} \right) = \xi _2  $,  and $ \tilde{P}\left( {y,{\xi  },\varepsilon } \right) =  - \varepsilon D\left( \varepsilon  \right)y_2 $, where $ D(x) $ denotes the Dirichlet function, i.e., $ D(x)= 1 $ when $ x $ is irrational, and $ D(x)= 0 $ when $ x $ is rational. Then one verifies that \ref{aph1} to \ref{aph3} automatically hold (with $ \varphi(\delta)=\mathcal{O}^\#(\delta) $), and the perturbation is small due to $ \left| {\varepsilon D\left( \varepsilon  \right)} \right| \leqslant \varepsilon  \to {0^ + } $ as $ \varepsilon \to 0^+ $, but the equations in \eqref{FPKAMlianli2} lead to 
\[\left\{ \begin{gathered}
	\xi _1^ *   = 0, \hfill \\
	\xi _2^ *  - \varepsilon D\left( \varepsilon  \right) = 0, \hfill \\ 
\end{gathered}  \right. \Rightarrow \left\{ \begin{gathered}
	\xi _1^ *  = 0, \hfill \\
	\xi _2^ *  = \varepsilon D\left( \varepsilon  \right), \hfill \\ 
\end{gathered}  \right.\]
which implies that the parameter $ \xi^* $ is almost everywhere discontinuous with respect to $ \varepsilon $ in any neighborhood of $ 0 $ (except for $ \varepsilon=0 $).

\subsubsection{Nowhere differentiability of $ \xi^* $ with respect to $ \varepsilon $}
Even if the obtained parameter $ \xi^* $ is continuous with respect to $ \varepsilon $, it might be nowhere differentiable. 

For example, let $ {{\tilde \omega }_1}\left( {\xi _1 } \right) = \xi _1  $, $ {{\tilde \omega }_2}\left( {\xi _2 } \right) = \xi _2  $,  and $ \tilde{P}\left( {y,{\xi },\varepsilon } \right) =   - \varepsilon W(\varepsilon)y_2 $, where $ W(x) $ denotes the Weierstrass function, e.g., $ W\left( x \right) = \sum\nolimits_{n = 0}^\infty  {{2^{ - n}}\cos \left( {{{99}^n}\pi x} \right)}  $.
It is easy to check that \ref{aph1} to \ref{aph3} automatically hold (with $ \varphi(\delta)=\mathcal{O}^\#(\delta) $), and the perturbation is small due to $\left| {\varepsilon W\left( \varepsilon  \right)} \right| \leqslant 2\varepsilon  \to {0^ + } $ as $ \varepsilon \to 0^+ $, but the equations in \eqref{FPKAMlianli2} yield 
\[\left\{ \begin{gathered}
	\xi _1^ *   = 0, \hfill \\
	\xi _2^ *  - \varepsilon W\left( \varepsilon  \right) = 0, \hfill \\ 
\end{gathered}  \right. \Rightarrow \left\{ \begin{gathered}
	\xi _1^ *  = 0, \hfill \\
	\xi _2^ *  = \varepsilon W\left( \varepsilon  \right), \hfill \\ 
\end{gathered}  \right.\]
which means that $ \xi^* $ is continuous but nowhere differentiable with respect to $ \varepsilon $.

\subsection{Utilizing extra parameters to adjust the Relative Singularity Condition \ref{aph2}}\label{FPKAMLESS}
 As we will show in this section,  relatively extra parameters sometimes play an unexpected role  to achieve frequency-preserving.

Let us first consider an interesting problem: If the parameter number of the frequency mapping is less than the perturbation, can there be a corresponding frequency-preserving KAM theorem like Theorem \ref{FPKAMT1}? Fortunately, we could  answer this question positively by establishing Theorem \ref{FPKAMT2} via weaker assumptions below. For the sake of simplicity in the statements of the main KAM results in Section \ref{FPKAMintro}, we prefer to present the weakened version here.
  
We still study the parameterized  perturbed Hamiltonian system \eqref{FPKAMHamilton} with continuous $ \omega \left(  \cdot  \right):{\mathcal{O}_1} \to {\mathbb{R}^n} $ and $ P\left( {y,x, \cdot } \right):{\mathcal{O}_1} \times {\mathcal{O}_2} \to {\mathbb{R}^1} $, where ${\mathcal{O}_1} \subset {\mathbb{R}^{{n}}} $ and $ {\mathcal{O}_2} \subset {\mathbb{R}^{{m}}} $ with $ n,m \in \mathbb{N}^+ $ are 
connected compact parameter sets having interior points. Denote by $ \xi  = (\bar \xi ,\tilde \xi ) \in {\mathcal{O}_1} \times {\mathcal{O}_2} $. Then the corresponding assumptions read:
  
\begin{enumerate}[label=(H\arabic*{*}), ref=(H\arabic*{*})]

\item \label{asp1x} [Internal Condition] For $ \Upsilon  \in \left(\omega(\mathcal{O}_1)\right)^o$ given in advance, there exists $ \xi_0 \in (\mathcal{O}_1)^o$ such that $ \Upsilon  = \omega \left( {{\xi _0}} \right) $ admits Diophantine nonresonance, which is defined as:
\begin{equation}\notag 
	\left| {\left\langle {k,\omega \left( {{\xi _0}} \right)} \right\rangle } \right| \geqslant \gamma {\left| k \right|^{ - \tau }}, \quad \gamma  > 0, \quad \tau>\max\{n-1,1\}, \quad \forall 0 \ne k \in {\mathbb{Z}^n}.
\end{equation}

\item \label{asp2x}  [Relative Singularity Condition] There exists a neighborhood $ \mathcal{V} \subset \mathcal{O}_1 $ of $ \xi_0 $ and a continuous function $ \tilde\xi (\bar\xi) $ with $\tilde\xi (\mathcal{V}) \subset (\mathcal{O}_2)^o $, such that the following holds (allowing the supremum to be continuously supplemented according to the sup-limit)
\[\mathop {\sup }\limits_{\bar\xi  \ne \bar\zeta ,\bar\xi ,\bar\zeta  \in \mathcal{V}} \frac{{\mathop {\sup }\limits_{y,x \in G \times {\mathbb{T}^n}} \left| {P\left( {y,x,\bar\xi, \tilde\xi (\bar\xi) } \right) - P\left( {y,x,\bar\zeta, \tilde\zeta (\bar\zeta) } \right)} \right|}}{{\left| {\omega ( \bar\xi  ) - \omega ( \bar\zeta  )} \right|}} <+\infty. \]
\end{enumerate}

Now we could give the following theorem.

\begin{theorem}\label{FPKAMT2}
Assume \ref{asp1x}, \ref{asp2x} and \ref{aph3}. Then there exists a sufficiently small $ {\varepsilon _0} > 0 $ such that, for any $ 0 < \varepsilon  < {\varepsilon _0} $,  one can find some $ {\xi^* }=(\bar\xi^*, \tilde\xi^*)\in ({\mathcal{O}_1} \times {\mathcal{O}_2})^o  $  such that the	perturbed Hamiltonian system $ H\left( {y,x,{\xi^* },\varepsilon } \right) $  admits an analytic quasi-periodic invariant torus with toral frequency $ \Upsilon=\omega \left( {{\xi _0}} \right) $. In particular, $  {\bar\xi^* } = {\xi _0} +\mathcal{O}\left(\varphi \left( \varepsilon  \right) { - \int_0^{\varepsilon } {\frac{{\varphi (x)}}{{x\ln x}}dx} } \right) $ as $ \varepsilon \to 0^+ $.
\end{theorem}

 It shall be mentioned that the  proof of Theorem \ref{FPKAMT2} is almost the same as that of Theorem \ref{FPKAMT1}, as long as we fix the extra parameters of the perturbation relative to the frequency mapping, then only the latter parameters are translated in each KAM step, see details from Section \ref{FPKAMProofT1}. 

Let us take  an explicit example to understand Theorem \ref{FPKAMT2} more clearly.  Consider the  parameterized  perturbed Hamiltonian system with parameter sets $ {\mathcal{O}_1} = {\left[ { - 2,2} \right]^2} $ and $ {\mathcal{O}_2} = \left[ { - 2,2} \right]^2 $:
\begin{equation}\label{FPKAMHAM2}
		H\left( {y,\xi,\varepsilon } \right) = \left\langle {\omega \left( \xi  \right),y} \right\rangle  + \varepsilon P\left( {y,\xi  } \right),
\end{equation}
where $ 0<\varepsilon\ll 1$, provided the frequency mapping $ \omega \left( \bar\xi  \right) = \left( {{{\bar \omega }_1} + \xi_1,{{\bar \omega }_2} + \xi_2} \right) $ with $ \left( {{{\bar \omega }_1},{{\bar \omega }_2}} \right) $ being Diophantine, and $P( {y,\xi })=\langle ( {{{\left| {{\xi _1}} \right|}^{1/2}} + {\xi _3},{{\left| {{\xi _2}} \right|}^{1/2}} + {\xi _4}} ), y\rangle     $. It is evident that \ref{aph2} fails near $ (0,0,0,0) $, but taking $ \tilde \xi  = \left( {{\xi _3},{\xi _4}} \right) = ( { - {{\left| {{\xi _1}} \right|}^{1/2}} + {\xi _1}, - {{\left| {{\xi _2}} \right|}^{1/2}} + {\xi _2}} ) $ one can verify that \ref{asp2x} holds near $ (0,0) $.  In other words, we can adjust the relative singularity by modifying the weak regularity with relatively extra parameters.

Following the same idea, we can also deal with the case that the frequency mapping has more parameters than the perturbation. Note that \ref{aph2} may not be satisfied at this time, but we have the opportunity to modify the relative singularity by doing similar approach so that \ref{asp2x} is satisfied. For example, consider the  parameterized  perturbed Hamiltonian system  \eqref{FPKAMHAM2}, where the frequency mapping $ \omega \left( \xi  \right) = \left( {{{\bar \omega }_1} + \xi_1,{{\bar \omega }_2} + \xi _2^2 + {\xi _3}} \right) $ with $ \left( {{{\bar \omega }_1},{{\bar \omega }_2}} \right) $ being  Diophantine, and $   P\left( {y,\xi } \right) = {\xi _2}{e^{{y_2}}} $. Here $ \xi_i \in [-2,2]$ for $ i=1,2,3 $. Then if $ y \in [-1,1]^2 $, taking $ \xi  = \left( {0,{\xi _2},0} \right) $ and $ \zeta  = \left( {0,{\zeta _2},0} \right) $ with $ {\xi _2} + {\zeta _2} = {j^{ - 1}} $ ($ j \in \mathbb{N}^+ $) yields that 
\begin{align*}
\mathop {\sup }\limits_{y,x \in G \times {\mathbb{T}^n}} \left| {P\left( {y,\xi } \right) - P\left( {y,\zeta } \right)} \right| &= \mathop {\sup }\limits_{{y_2} \in \left[ { - 1,1} \right]} {e^{{y_2}}}\left| {{\xi _2} - {\zeta _2}} \right| = e\left| {{\xi _2} - {\zeta _2}} \right|\\
& = ej\left| {\xi _2^2 - \zeta _2^2} \right| = ej\left| {\omega \left( \xi  \right) - \omega \left( \zeta  \right)} \right|,
\end{align*}
which contradicts \ref{aph2} near $ (0,0,0) $, whenever $ j $ is sufficiently large. However, taking $ \xi_3 $ as a function of $ \xi_2 $ may make  \ref{asp2x} hold near $ (0,0) $ with the new  parameter $ \xi=(\xi_1,\xi_2) $ ($ \zeta=(\zeta_1,\zeta_2) $) of  two variables, e.g., $ \xi_3:=- \xi_2^2+\xi_2$. Then one can similarly achieve  frequency-preserving.

It shall be mentioned that this approach is also valid for partial frequency-preserving KAM in Section \ref{FPKAMPAR} and infinite-dimensional KAM in Section \ref{FPKAMINFINITE}.

\subsection{Partial frequency-preserving KAM}\label{FPKAMPAR}
As known to all, in general, some dynamics cannot be completely preserved under the perturbation if the unperturbed dynamical system has certain degeneracy. Back to our concern on the frequency-preserving KAM, it is evident that the degeneracy of the Hamiltonian system may destroy the prescribed Diophantine frequency for the perturbed torus, even in the parameterized settings (see the counterexample constructed in Section \ref{FPKAMSEC217}). Partial preservation of frequencies, therefore, is a foundational problem in KAM theory, although remains relatively unexplored. On this aspect, via a nondegenerate condition of R\"ussmann type on a submanifold \cite{MR1843664},  Chow et al. \cite{MR1938331} proved that the perturbed toral frequencies could be partially preserved according to the maximal degeneracy of the Hessian matrix of the unperturbed Hamiltonian system. They also obtained KAM results concerning partial frequency-ratio-preserving. See also Sevryuk \cite{MR2221801,MR2433684,MR3783834} for partial frequency-preserving KAM  under moderately degenerate integrable or partially integrable Hamiltonian systems. To be more precise, he proved that the unperturbed invariant $ n $-tori with prescribed frequencies (or frequency ratios) do not persist, but the first $ d<n $ frequencies (or ratios) are preserved. We also mention the recent work of Zhao and Li \cite{MR4355926}, which considers the partial frequency-ratio-preserving KAM  for unperturbed Hamiltonian systems  starting from the same Riemannian manifold.

Interestingly, a new version of the  partial frequency-preserving KAM theorem could also be obtained via a similar approach in this paper. Note that the nondegeneracy here is mainly reflected in the requirement of the frequency mapping $ \omega(\xi) $, namely the Relative Singularity Condition \ref{aph2} and the  Controllability Condition \ref{aph3}. As a consequence, one may consider the case where \ref{aph2} holds for some (not all) components of $ \omega(\xi) $, i.e., $ \omega(\xi) $ admits certain degeneracy. In this case, we could  combine a succession of our parameter translation technique with the classic digging method to establish a partial frequency-preserving KAM theorem. Let us consider the simplest case, i.e., the parameterized Hamiltonian system \eqref{FPKAMHamilton} with an uncoupled frequency mapping $ \omega(\xi)=(\omega_1(\xi_1), \cdots, \omega_n(\xi_n)) $ having \textit{nondegeneracy} for components of indices $ 1 \leqslant {i_1} <  \cdots  < {i_a} \leqslant n $. Denote by $ 1\leqslant{\ell _1}<  \cdots  <{\ell _b}\leqslant n $ the other indices (with \textit{nondegeneracy}), where $ a+b  = n $ and $ a<b $. On these grounds, let $ \hat\xi=(\xi_{i_1},...,\xi_{i_a}) $ and $ \check\xi=(\xi_{\ell_1},...,\xi_{\ell_b}) $ for convenience. %Then we assume that $ P_{i_j}=P_{i_j}(y,x,\xi_{{i_j}}) $ for $ 1 \leqslant j \leqslant a $. Note that the remaining components of $ P $ allow parametric coupling. 
Now,  the corresponding new assumptions are given below.

\begin{enumerate}[start=2, label=(H\arabic*${}^\prime$), ref=(H\arabic*${}^\prime$)]
	
\item \label{hyp:h2_prime}  [Relative Singularity Condition] There exists a neighborhood $ \mathcal{V}' \subset  \prod\nolimits_{j = 1}^a {{\mathcal{O}_{{i_j}}}}  $ of $ \hat\xi_0=({({\xi _0})_{{i_1}}}, \ldots, {({\xi _0})_{{i_a}}}) $, such that  the following holds (allowing the supremum to be continuously supplemented according to the sup-limit)
\[\mathop {\sup }\limits_{\hat\xi  \ne \hat\zeta ,\hat\xi ,\hat\zeta  \in \mathcal{V}'} \frac{{\mathop {\sup }\limits_{y,x \in G \times {\mathbb{T}^n},\;\check \xi \in \prod\nolimits_{j = 1}^b {{\mathcal{O}_{{\ell _j}}}} } \left| {P\left( {y,x,\check\xi,\hat \xi } \right) - P\left( {y,x,\check\xi,\hat \zeta } \right)} \right|}}{\sum\limits_{j = 1}^a { {\left| {\omega_{i_j} \left( \xi_{i_j}  \right) - \omega_{i_j} \left( \zeta_{i_j}  \right)} \right|}}}<+\infty.\]
Moreover, $ |\omega(\xi)-\omega(\zeta)| \geqslant L|\xi-\zeta| $ with some $ L>0 $ and all $ \xi,\zeta \in \mathcal{O} $.

\item \label{hyp:h3_prime} [Controllability Condition] Assume that 
\[ \varphi \left( \delta  \right): = \sup_{1\leqslant j \leqslant a}\left\{ { \left| {\xi_{i_j}  - \zeta_{i_j} } \right|:\left| {\omega_{i_j} \left( \xi_{i_j}  \right) - \omega_{i_j} \left( \zeta_{i_j}  \right)} \right| \leqslant \delta } \right\} \]
 is continuously defined on $ [0, {\tau'}] $ with some $ 0<{\tau'}\leqslant1/2 $, and satisfies $ \varphi(0)=0 $ and $ \varphi(x)>0 $ for $ x \in  (0, {\tau'}] $. Moreover, it holds
\[ - \int_0^{{{\tau'}}} {\frac{{\varphi (x)}}{{x\ln x}}dx}  <  + \infty .\]
\end{enumerate}

\begin{theorem}[Partial frequency-preserving KAM]\label{PFPKAM}
Assume \ref{aph1}, \ref{hyp:h2_prime} and \ref{hyp:h3_prime}. Then there exists a sufficiently small $ {\varepsilon _0} > 0 $ such that, for any $ 0 < \varepsilon  < {\varepsilon _0} $,  one can find some $ {\xi^* }  $ in a family of Cantor sets $ {\Pi _\varepsilon }: = \prod\nolimits_{j = 1}^b {\left( {{\mathcal{O}_{{\ell _j}}}\backslash {\mathcal{O}_{{\ell _j}}}\left( \varepsilon  \right)} \right)}  \times \left\{ {\xi _{{i_1}}^ * , \ldots ,\xi _{{i_a}}^ * } \right\} \subset \mathcal{O}^o$,  such that the	perturbed Hamiltonian system $ H\left( {y,x,{\xi^* },\varepsilon } \right) $ with parameter $ \xi^* $  admits an analytic quasi-periodic invariant torus with toral  frequency $ \omega(\xi^*) $, where $ \omega_{i_j}(\xi^*)=\omega_{i_j} \left( {{\xi _0}} \right) $ for all $ 1 \leqslant j \leqslant a $. Moreover, $ \mathop {\lim }\nolimits_{\varepsilon  \to {0^ + }} (\xi _{{i_1}}^ * , \ldots ,\xi _{{i_a}}^ * ) = ({({\xi _0})_{{i_1}}}, \ldots ,{({\xi _0})_{{i_a}}}) $, and $ \mathop {\lim }\nolimits_{\varepsilon  \to {0^ + }} {\rm meas}\prod\nolimits_{j = 1}^b {\left( {{\mathcal{O}_{{\ell _j}}}\backslash {\mathcal{O}_{{\ell _j}}}\left( \varepsilon  \right)} \right)}  = {\rm meas}\prod\nolimits_{j = 1}^b {{\mathcal{O}_{{\ell _j}}}}  $.
\end{theorem}
\begin{remark}
	Note that both the frequency mapping $ \omega $ and the perturbation $ P $ could still have weak regularity with respect to the parameter, e.g., the H\"older continuity and even the Logarithmic H\"older continuity, see  Comment \ref{comc5} and  Remark \ref{FPKAMRE1.2}.
\end{remark}

 For the sake of brevity, let us only give the key idea. We also recommend that the reader skip here and come back to it after following the whole proof of  Theorem \ref{FPKAMT1} in Sections \ref{FPKAMsec-4} and \ref{FPKAMProofT1}. In this case, the semi-norm for $ P $ similar to \eqref{FPKAMSEMI} should be sightly modified according to the Relative Singularity Condition \ref{hyp:h2_prime}.  We shall first fix the $ 1 \leqslant {i_1} <  \cdots  < {i_a} \leqslant n $ components of the parameter $ \xi_\nu $ in the $ \nu $-th KAM step, and dig small measures for other components, i.e., indices $ 1\leqslant{\ell _1}<  \cdots  <{\ell _b}\leqslant n $. This is indeed classic in the traditional KAM theory, see Chow et al. \cite{MR1938331}, Li and Yi \cite{MR1926285,MR2003447}, Zhao and Li \cite{MR4355926} and etc. for relevant technique and iteration settings (note that the property $ |\omega(\xi)-\omega(\zeta)| \geqslant L|\xi-\zeta| $ in the Relative Singularity Condition \ref{hyp:h2_prime} ensures this, and we may dig less measures if $ \omega(\xi) $ has weaker regularity). Then we use the Internal Condition \ref{aph1} to find the new parameter $ \xi_{\nu+1} $ for partial frequency-preserving, i.e.,
 \begin{equation}\label{FPKAMpartial}
 	{\omega _{{i_j}}}({({\xi _{\nu  + 1}})_{{i_j}}}) + \sum\limits_{l = 0}^\nu  {{{(p_{01}^l)}_{{i_j}}}({\xi _{\nu  + 1}})}  = {\omega _{{i_j}}}({\xi _0}), \quad 1 \leqslant j \leqslant a.
 \end{equation}
Now, similar to \eqref{FPKAMCANSHU5}, the Relative Singularity Condition \ref{hyp:h2_prime} and the Controllability Condition \ref{hyp:h3_prime} will provide the estimates for $ ({({\xi _{\nu+1} })_{{i_1}}}, \ldots ,{({\xi _{\nu+1} })_{{i_a}}}) $, i.e., 
\[\left| {{{({\xi _{\nu  + 1}})}_{{i_j}}} - {{({\xi _\nu })}_{{i_j}}}} \right| \leqslant \varphi ({\varepsilon ^{{q^\nu }}}), \quad 1 \leqslant j \leqslant a\]
with some $ q>1 $ and the function $ \varphi $ in the Controllability Condition \ref{hyp:h3_prime}, because the KAM errors are still super-exponentially small. Moreover, the integrability in the Controllability Condition \ref{hyp:h3_prime} enables us to prove that the partial parameter sequence $ {\left\{ {({{({\xi _\nu })}_{{i_1}}}, \ldots ,{{({\xi _\nu })}_{{i_a}}})} \right\}_{\nu  \in {\mathbb{N}^ + }}} $ is indeed a Cauchy sequence, therefore we can denote  by $ (\xi _{{i_1}}^ * , \ldots ,\xi _{{i_a}}^ * ) $ the corresponding limit (with respect to $ \nu $). Obviously, it tends to $ ({({\xi _0})_{{i_1}}}, \ldots ,{({\xi _0})_{{i_a}}}) $ as $ \varepsilon \to 0^+ $. Finally, after infinitely many KAM steps, the parameter set for indices $ {\ell _1}, \ldots ,{\ell _b} $ becomes $ \prod\nolimits_{j = 1}^b {\left( {{\mathcal{O}_{{\ell _j}}}\backslash {\mathcal{O}_{{\ell _j}}}\left( \varepsilon  \right)} \right)} $ with asymptotically full Lebesgue measure, i.e., $ \mathop {\lim }\nolimits_{\varepsilon  \to {0^ + }} {\rm meas}\prod\nolimits_{j = 1}^b {\left( {{\mathcal{O}_{{\ell _j}}}\backslash {\mathcal{O}_{{\ell _j}}}\left( \varepsilon  \right)} \right)}  = {\rm meas}\prod\nolimits_{j = 1}^b {{\mathcal{O}_{{\ell _j}}}}  $. We therefore obtain the KAM torus with  partial frequency-preserving due to the equations in  \eqref{FPKAMpartial}, i.e., $ ({({\xi _0})_{{i_1}}}, \ldots ,{({\xi _0})_{{i_a}}}) $ in the toral frequency is preserved.

To see the above process more clearly, we provide a simple example when $ n=2 $, $ i_1=1 $ and $ \ell _1=1 $. Now, the Relative Singularity Condition \ref{hyp:h2_prime} holds for $ i_1=1 $, i.e., we aim to preserve the first component $ \omega_1(\xi_0) $ of the prescribed Diophantine frequency $ \omega(\xi_0)=(\omega_1(\xi_0),\omega_2(\xi_0)) $. Therefore, during our KAM iteration, \eqref{FPKAMpartial} becomes
\begin{equation}\label{FPKAMC2}
	{\omega _{{1}}}({{(\xi_{\nu+1})_1}}) + \sum\limits_{l = 0}^\nu  {{{(p_{01}^l)}_{{1}}}(({\xi _{\nu  + 1}})_1, ({\xi _{\nu  + 1}})_2)}  = {\omega _{{1}}}({\xi _0})
\end{equation}
for $ \xi_{\nu  + 1}=((\xi_{\nu  + 1})_1,(\xi_{\nu  + 1})_2) $. After digging out some domain when $ ({\xi_{\nu+1}})_1 $ is fixed, we could solve \eqref{FPKAMC2} by appropriately choosing a new $ ({\xi_{\nu+1}})_1 $ with fixed $ ({\xi_{\nu+1}})_2 $ in the rest domain, i.e., the first component of $ \omega(\xi_0) $ is always unchanged.

It should be pointed out that, the assumptions proposed in  our partial frequency-preserving KAM also admit indispensability, since one can construct counterexamples similar to those in Section \ref{FPKAMCONTEREXAMPLES}. 

\subsection{Applicability for infinite-dimensional Hamiltonian systems}\label{FPKAMINFINITE}
As mentioned in \cite{MR4731278}, this parameter translation technique is not dimensionally limited, e.g., the Internal Condition  \ref{aph1} can be directly generalized to the infinite-dimensional case ($ n=+\infty $). Therefore, one can use the detailed spatial structure of P\"oschel \cite{MR1037110} or Montalto and Procesi \cite{MR4201442} to establish the corresponding infinite-dimensional frequency-preserving KAM theorem via irregular continuity with respect to the parameter, see \cite{MR4731278,DTL} for instance. In particular,  the frequency mapping may not possess spectral asymptotic properties \cite{DTL}.
However, as mentioned in the Introduction, the most significant challenge currently is the difficulty in constructing explicit and applicable systems that are \textit{strictly weaker} than the  Lipschitz type, and the frequency mapping is difficult to achieve very high smoothness. Fortunately, the general approach based on relative singularity and controllability introduced in this paper essentially overcomes these difficulties. See also Comments \ref{comc1},  \ref{comc3}, \ref{comc5} and Remark \ref{FPKAMRE1.2} for further details. We will present \textit{for the first time} below an infinite-dimensional KAM theorem with explicit arbitrarily weak regularity, which guarantees the existence of the full-dimensional tori with frequency-preserving, without requiring any spectral asymptotics.

Let us take the spatial structure introduced by P\"oschel in  \cite{MR1037110} for instance. Let $ \Lambda $ be an infinite-dimensional lattice with a weighted spatial structure $ \mathcal{S} $, where $ \mathcal{S} $ is a family of finite subsets $ A $ of $ \Lambda $. Namely, $ \mathcal{S} $ is a spatial structure on $ \Lambda $ characterized by the property that the union of any two sets in $ \mathcal{S} $ is again in $ \mathcal{S} $, if they intersect:
\[A,B \in \mathcal{S}, \quad A \cap B \ne \phi  \Rightarrow A\cup B  \in \mathcal{S}.\]
Then, we introduce a nonnegative weight function $ \left[  \cdot  \right]:A \to \left[ A \right] $ defined on $ \mathcal{S}\cap \mathcal{S}  = \left\{ {A\cap B :A,B \in \mathcal{S}} \right\} $ to reflect the size, location and something else of the set $ A $. The weight function satisfies the monotonicity and subadditivity for all $ A,B $ in $ \mathcal{S} $:
\begin{align*}
	A \subseteq B &\Rightarrow \left[ A \right] \leqslant \left[ B \right],\\
	A \cap B \ne \phi  &\Rightarrow \left[ {A \cup B} \right] + \left[ {A \cap B} \right] \leqslant \left[ A \right] + \left[ B \right].
\end{align*}
Next, we define the norms for $ k $ runs over all nonzero
integer vectors in $ \mathbb{Z}^\Lambda $ whose support $ \operatorname{supp} k = \left\{ {\lambda :{k_\lambda } \ne 0} \right\} $ is a finite set:
\[\left| k \right|: = \sum\limits_{\lambda  \in \Lambda } {\left| {{k_\lambda }} \right|} , \quad \left[ {\left[ k \right]} \right] = \mathop {\min }\limits_{\operatorname{supp} k \subseteq A \in \mathcal{S}} \left[ A \right].\]
Under these grounds, the infinite-dimensional nonresonant condition can be defined as follows.	
\begin{definition}[Infinite-dimensional nonresonant condition]\label{FPKAMifnc}
	Given a nondecreasing approximation function $ \Delta :\left[ {0, + \infty } \right) \to \left[ {1, + \infty } \right) $, i.e.,  $ \Delta \left( 0 \right) = 1 $, and
	\[\frac{{\ln \Delta \left( t \right)}}{t} \searrow 0  \quad \text{as \;$ 0 \leqslant t \to  + \infty $},  \quad \int_1^{ + \infty } {\frac{{\ln \Delta \left( t \right)}}{t}dt}  <  + \infty .\]
	Then for some $ \alpha>0 $ and every $ 0 \ne k \in {\mathbb{Z}^\Lambda } $ with finite support, the infinite-dimensional nonresonant condition is defined as
	\[\left| {\left\langle {k,\omega } \right\rangle } \right| \geqslant \frac{\alpha }{{\Delta \left( {\left[ {\left[ k \right]} \right]} \right)\Delta \left( {\left| k \right|} \right)}}.\]
	
\end{definition}
Let $ N =  \left\langle {\omega(\xi) ,y} \right\rangle  $ be the unperturbed integrable Hamiltonian with the frequency mapping $ \omega \left( \xi  \right):{\mathbb{R}^\Lambda } \supseteq \mathcal{Z}  \to {\mathbb{R}^\Lambda } $, and $ \varepsilon P=\varepsilon P(y,x,\xi) $ be the perturbation of the form $  \varepsilon\sum\nolimits_{A \in \mathcal{S}} {{P_A}\left( {{y_A},{x _A};{\xi_A }} \right)}  $, where $ \varepsilon>0 $ is sufficiently small, and $ {y _A} = \left( {{y _\lambda }:\lambda  \in A} \right) $, and similarly $ x_A $ and $ \xi_A $. Suppose that the perturbed Hamiltonian
\begin{equation}\label{FPKAMinfinteH}
	H = N + \varepsilon P =  \left\langle {\omega(\xi),y} \right\rangle  + \varepsilon\sum\limits_{A \in S} {{P_A}\left( {{y_A},{x _A};{\xi_A}} \right)}	
\end{equation}
is real analytic in the phase space variables $ y,x $ on a complex neighbourhood
\[{\mathcal{D}_{r,s}}: {\left| y \right|_w} < s, \quad {\left| {\operatorname{Im} x } \right|_\infty } < r\]
of the torus $ {\mathcal{T}_0}: =  \left\{ 0 \right\}\times {\mathbb{T}^\Lambda }  $ with $ w>0 $, and continuous in the parameter $ \xi $ on $ \mathcal{Z} $.  The corresponding norms here are defined by
\[{\left| y \right|_w} := \sum\limits_{\lambda  \in \Lambda } {\left| {{y_\lambda }} \right|{e^{w\left[ \lambda  \right]}}} , \quad \left[ \lambda  \right] := \mathop {\min }\limits_{\lambda  \in A \in \mathcal{S} \cap \mathcal{S}} \left[ A \right], \quad {\left| x  \right|_\infty } := \mathop {\sup }\limits_{\lambda  \in \Lambda } \left| {{x _\lambda }} \right|.\]
Moreover, the size of the perturbation $ P $ is measured in terms of the
weighted norm
\[|||P|||{_{m,r,s}}: = \sum\limits_{A \in S} {{{\left\| {{P_A}} \right\|}_{r,s}}{e^{m\left[ A \right]}}} ,\]
where $ {P_A} = \sum\nolimits_k {{P_{A,k}}\left( {y,\xi } \right){e^{\sqrt { - 1} \left\langle {k,x } \right\rangle }}}  $ is the Fourier series expansion, and
\[\|{P_A}\|{_{r,s}}: = \sum\limits_{k \in {\mathbb{Z}^\Lambda }} {{{\left\| {{P_{A,k}}} \right\|}_s}{e^{r\left| k \right|}}} , \quad {\left\| {{P_{A,k}}} \right\|_s}: = \mathop {\sup }\limits_{{{\left| y \right|}_w} < s,\;\xi  \in \mathcal{Z}} \left| {{P_{A,k}}\left( {y,\xi } \right)} \right|.\]
Since quantitative estimates of the smallness of the  perturbation are not emphasized here, we omit some additional notations.
Similar to Theorem \ref{FPKAMT1} considering the finite-dimensional case, we make the following assumptions.

 \begin{enumerate}[label=(H\arabic*${}^{\text{\#}}$), ref=(H\arabic*${}^{\text{\#}}$)]

\item \label{hyp:h1_hash} [Internal Condition] For $ \Xi   \in \left(\omega(\mathcal{Z})\right)^o$ given in advance, there exists $ \xi_0 \in \mathcal{Z}^o$ such that $ \Xi  = \omega \left( {{\xi _0}} \right) $ admits the infinite-dimensional nonresonant condition in Definition \ref{FPKAMifnc}.

\item \label{hyp:h2_hash} [Relative Singularity Condition] There exists a neighborhood $ \mathcal{U} \subset \mathcal{Z} $ of $ \xi_0 $, such that the following holds (allowing the supremum to be continuously supplemented according to the sup-limit)
\[\mathop {\sup }\limits_{\xi  \ne \zeta ,\xi ,\zeta  \in \mathcal{U}} \frac{ ||| {P\left( {y,x,\xi } \right) - P\left( {y,x,\zeta } \right)} |||_{m,r,s}}{{\left| {\omega \left( \xi  \right) - \omega \left( \zeta  \right)} \right|_\infty}} <+\infty. \]
\end{enumerate}

  The framework for the frequency-preserving  infinite-dimensional KAM theorem is somewhat similar to that in Sections \ref{FPKAMsec-4} and \ref{FPKAMsec-5}. Instead of digging out domains, we use the technique of parameter translation to keep the prescribed nonresonant frequency $ \Xi   $ unchanged, and consequently, the measure estimates are not involved along our KAM approach. To be more precise, for some $ {{ \xi }_0} = {\omega ^{ - 1}}\left( {\Xi} \right) \in \mathcal{Z}  \subseteq {\mathbb{R}^\Lambda } $,  we could obtain by the Internal Condition \ref{hyp:h1_hash} a parameter sequence $ {\{ {{{ \xi }_\nu }} \}_{\nu \in \mathbb{N}^+} } $ near $ {{ \xi }_0} $ due to the KAM smallness through iteration and the continuity of $ \omega $, such that $ \omega( \xi_\nu)=\Xi $, i.e., we achieve frequency-preserving. Except for this, similar to \eqref{FPKAMCANSHU6}, $ {\{ {{{ \xi }_\nu }} \}_{\nu \in \mathbb{N}^+} } $ could be proved to be a  Cauchy sequence in the norm $ | \cdot |_\infty $ by applying the Relative Singularity Condition \ref{hyp:h2_hash}. Since below we only consider the case that the perturbation possesses \textit{at most} parametric Lipschitz continuity for brevity, we do not need an  assumption like the Controllability Condition \ref{aph3} according to Comment \ref{comc4}. Finally, the uniform convergence of the KAM iteration could be obtained as usual. We therefore give the following infinite-dimensional frequency-preserving KAM theorem without proof. As for the detailed proof based on the (uniform) weak convexity condition in the infinite-dimensional setting, although it is subject to more constraints, we recommend that readers refer to Du et al.  \cite{DTL}.

\begin{theorem}[Infinite-dimensional frequency-preserving KAM]\label{th23}
	Consider the Hamiltonian system \eqref{FPKAMinfinteH} with the perturbation having at most Lipschitz continuity with respect to the parameter $ \xi $. Assume that  \ref{hyp:h1_hash} and \ref{hyp:h2_hash} hold. Then there exists at least a $ \xi^*  $ near $ \xi_0 $ as long as $ \varepsilon>0 $ is sufficiently small, such that \eqref{FPKAMinfinteH} admits an  analytic, quasi-periodic, full-dimensional invariant torus with toral frequency $ \Xi=\omega \left( {{\xi _0}} \right)$. %In particular, $  {\xi ^ * } = {\xi _0} +\mathcal{O}\left( { - \int_0^{\varepsilon } {\frac{{\varphi (x)}}{{x\ln x}}dx} }\right) $ as $ \varepsilon \to 0^+ $.
\end{theorem}

As we previously mentioned, our Theorem \ref{th23} has an extremely broad scope of application. Next, we provide an explicit example (though not limited to this weakly coupled one) to facilitate the reader's understanding and application.

Assume that the components of the frequency mapping $ \omega(\xi) $ in the Hamiltonian \eqref{FPKAMinfinteH} take the form  
\begin{align}\notag 
	{\omega ^i}\left( \xi  \right) = {\varphi _{i,0}} + {\varphi _{i,1}}{{\tilde \omega}^{i,1}}\left( \xi  \right) +  \cdots  = {\varphi _{i,0}} + \sum\limits_{j = 1}^\infty {{\varphi _{i,j}}{{\tilde \omega}^{i,j}}\left( \xi  \right)} , \quad  i \in \mathbb{N}^+,\end{align}
where $ {\inf _{i \in {\mathbb{N}^ + }}}{\varphi _{i,0}} > 0 $. In addition, set
\begin{equation}\notag
	\sum\limits_{j = 1}^\infty {\frac{{{\varphi _{i,j}}}}{{{\varphi _{i,0}}}}{{\tilde \omega}^{i,j}}\left( \xi  \right)}  := {{\bar { \omega}}^i}(\xi), \quad \left( {{{\bar \omega}^1},{{\bar \omega}^2}, \ldots ,{{\bar \omega}^i}, \ldots } \right) := \bar \omega(\xi),
\end{equation}
where $ \bar{\omega} $ is assumed to be a local homeomorphism near $ \Xi $, with at most Lipschitz continuity (e.g., H\"older continuity or even weaker, see Comment \ref{comc5} and Remark \ref{FPKAMRE1.2}). Finally, assume that
\[\sum\limits_{j = 0}^\infty {\left| {{\varphi _{i,j}}} \right|}  = \mathcal{O}\left(e^{w'[i]}\right), \quad 0<w'<w, \quad i \to  + \infty .\]
It then readily follows that the components  $ \omega^i(\xi) $ cannot exceed the exponential growth $ \mathcal{O}\left(e^{w[i]}\right) $, hence the inner product $ \left\langle {\omega(\xi) ,y} \right\rangle  $ is well-defined for $ {\left| y \right|_w} < s $. Note that in this case, we do not impose any spectral asymptotics, such as $ \omega^i(\xi) \sim i^d  $ for some $ d \in \mathbb{N} $ as $ i \to +\infty $, but merely require that the modulus of its components do not grow too rapidly. Now, assume that the perturbation $ P $ in the Hamiltonian \eqref{FPKAMinfinteH} also admits at most Lipschitz regularity and that \ref{hyp:h1_hash} and \ref{hyp:h2_hash} hold, where $ \omega(\xi) $ in \ref{hyp:h2_hash} is replaced by $ \bar{\omega}(\xi) $. It then readily follows  that Theorem \ref{th23} holds (or one can refer to the proof framework in \cite{DTL}), which implies the existence of the full-dimensional torus with  frequency-preserving. Note that under these grounds, the regularity of both the frequency mapping $ \omega(\xi) $ and the perturbation $ P $ can be arbitrarily weak and can be explicitly specified, therefore Theorem \ref{th23} can be applied to many concrete problems.

 In addition to the above, many counterexamples in the infinite-dimensional case can be constructed similarly. Indeed, one just needs to adjust the infinite-dimensional part appropriately and emphasize the finite-dimensional part. See Sections \ref{FPKAMSEC2.1} and \ref{FPKAMSEC2.2} for further details.

\subsection{Further comments on the Controllability Condition \ref{aph3}} 
We conjecture that the Controllability Condition \ref{aph3} is independently indispensable due to the cumulative effect (the nonindependent indispensability can be seen in Section \ref{FPKAMSEC217}). However, we do not have a rigorous counterexample to show this. Here we give an intuitive idea. Consider the parameterized Hamiltonian system \eqref{FPKAMHamilton}, where the perturbation is independent of the parameter. Then, using the parameter translation technique, we obtain the frequency-preserving equations $ 	- {p_{01}^\nu }  = \omega \left( {{\xi _ {\nu+1} }} \right) - \omega \left( \xi_\nu  \right)$ for all $ \nu \in \mathbb{N} $ (see Section \ref{FPKAMProofT1}). We assume that all $ p_{01}^\nu =- a_\nu\cdot (1, \ldots ,1):=-a_\nu I $, where $ \{a_\nu\}_{\nu\in \mathbb{N}}$ is a strictly decreasing positive sequence. Note that this is not true in all cases, because the drift terms in KAM theory may not hold signs, but we can choose special perturbations that make it true. Then one may construct some frequency mapping such that $ \xi_\infty =+\infty$, in other words, after some $ N \gg 1 $, one cannot find a parameter $ \xi_{N+1} \in \mathcal{O} $ in the sense of frequency-preserving. To be more precise, let $ \Delta \in (0, \min \left\{ {1,{\rm{diam}}\mathcal{O}/2} \right\} ) $ and $ n \in \mathbb{N}^+ $ be given. Then, we construct a $ C^\infty $  frequency mapping $ \omega(\xi) $ on $ \mathbb{R}^n $ that satisfies
\[\left\{ \begin{gathered}
	\omega \left( {{\xi _0}} \right) = \bar \omega , \quad  {\text{$ \bar \omega $ is Diophantine}}, \hfill \\
	\omega \left( {{\xi _0} + \nu \Delta I } \right) = \bar \omega  + \sum\nolimits_{j = 0}^{\nu  - 1} {{a_j}} I, \quad \forall\nu  \in {\mathbb{N}^ + }, \hfill \\
	{D^u}\omega \left( {{\xi _0} + \nu \Delta I } \right) = 0, \quad \forall \nu ,u \in \mathbb{N}^+, \hfill \\
	\omega \left( {{\xi _0} + s} \right) \ne \omega \left( {{\xi _0} + \nu \Delta I } \right), \quad {\text{$ \forall s $ between $ \left( {\nu  - 1} \right)\Delta I $ and  $ \nu \Delta I ,  \quad \nu  \in {\mathbb{N}^ +  }$}}. \hfill \\ 
\end{gathered}  \right.\]
  It is evident to verify that the $ \bar \omega $-frequency-preserving equations give $ {\xi _\nu } = {\xi _0} + \nu \Delta I$ for all $ \nu \in \mathbb{N}^+ $.  If we limit $ \omega(\xi) $ on $ \mathcal{O} $, then there exists a sufficiently large $ N \in \mathbb{N}^+$ such  that $ \xi_{\nu} \notin \mathcal{O} $ for all $ \nu \geqslant N+1$ (this also shows that \ref{aph3} fails). Even if $ \mathcal{O}=\mathbb{R}^n $, $ \{\xi_{\nu}\}_{\nu \in \mathbb{N}^+} $ does not converge and therefore cannot technically reach  frequency-preserving. The above analysis gives an intuitive conclusion, but it seems difficult to retroactively construct an explicit parameterized Hamiltonian system to illustrate this point.

We also believe that the integrability in the Controllability Condition \ref{aph3} is strongly related to the irrationality of the prescribed frequency (e.g., the Diophantine nonresonance in  \eqref{FPKAMDIO}) and the regularity of the perturbed Hamiltonian system with respect to the action-angle variables (e.g., the analyticity in this paper). This is because the  integrability plays a crucial role in our proof, showing that the parameter sequence is indeed a Cauchy sequence through the super-exponential decay of KAM errors. However, if one considers frequencies that are nearly rational, such as the weak Diophantine type (see Tong and Li \cite{MR4836959}, etc.) and the Bruno type (see Han\ss mann and Si \cite{MR2586370}, etc.), or Hamiltonian systems that are less regular with respect to the action-angle variables, such as non-analytic Gevrey regularity (see Popov \cite{MR2104602},  Lopes Dias and Gaiv\~{a}o \cite{MR4011043},	Chen and Cheng \cite{MR4208447},  etc.), or even finitely smooth regularity (see P\"{o}schel \cite{MR668410}, Salamon \cite{MR2111297}, Li and Shang \cite{MR3960504}, Koudjinan \cite{MR4104457}, Tong and Li \cite{CCM}, etc.), then the KAM convergence rate may be slower (exponentially slow convergence, or even arbitrarily slow convergence), which makes us have to require some  stronger integrability in \ref{aph3} via our approach.

\section{The KAM step}\label{FPKAMsec-4}
In this section, we present the basic lemmas related to our KAM theorem in detail, for completeness. Among these, \textit{the most critical} is Lemma \ref{FPKAMcrucial}, which is essential for frequency-preserving. 
It should be emphasized that we  employ here the quasi-linear KAM iterative scheme instead of the linear one, which was first introduced by Li and Yi in \cite{MR2003447}, and later utilized in \cite{MR4669322,MR4731278,MR4806284}, etc. This technique could overcome some essential difficulties that cannot be solved by the classic KAM iteration. Although it might not be necessary under the settings in this paper, we prefer to follow this approach which may lead ones directly extend such KAM results to some other fundamental problems. 

Let $ {\rm{diam}}\mathcal{O} \leqslant 1 $ without loss of generality. 	For a perturbation function $ P\left( {y,x,\xi } \right) $ (it may vary in the KAM step), which is analytic with respect to $ y $ and $ x $ on  some closed domain $ D $ and only continuous with respect to $ \xi $  on $\mathcal{V}\subset \mathcal{O} $, we define its norm as follows
\[{\left\| P \right\|_D}: = {\left| P \right|_D} + {\left[ P \right]_{\omega}},\]
where the sup-norm is given by
\[{\left| P \right|_D} := \mathop {\sup }\limits_{\xi  \in \mathcal{V}} \mathop {\sup }\limits_{\left( {y,x} \right) \in D} \left| P \right|,\]
and the semi-norm is defined as 
\begin{equation}\label{FPKAMSEMI}
		[P]_\omega=:\mathop {\sup }\limits_{\xi  \ne \zeta ,\xi ,\zeta  \in \mathcal{V}} \frac{{\mathop {\sup }\limits_{y,x \in G \times {\mathbb{T}^n}} \left| {P\left( {y,x,\xi } \right) - P\left( {y,x,\zeta } \right)} \right|}}{{\left| {\omega \left( \xi  \right) - \omega \left( \zeta  \right)} \right|}}.
\end{equation}
Note that this $ \omega $-dependent semi-norm is well defined thanks to the Relative Singularity Condition \ref{aph2} and the Controllability Condition \ref{aph3} introduced in this paper. Unlike the H\"older semi-norm in \cite{DL} or the weaker one based on the modulus of continuity in \cite{MR4731278,DTL}, this semi-norm still allows the KAM iteration, while significantly reducing the regularity requirements for both the frequency mapping and the perturbation, with respect to the parameter.

\subsection{Description of the $ 0 $-th KAM step}
Set $\rho  = 1/10,\gamma  = {\varepsilon ^{1/4(n+6)}} $ and let $ \eta  > 0 $ be an integer such that $ {\left( {1 + \rho } \right)^\eta } > 2 $.  We first define the following parameters of the $ 0 $-th KAM step:
\begin{align}
	\label{FPKAMCANSHU}&{r_0} = r,\;{\gamma _0} = \gamma ,\;{e_0} = 0,\;{\bar h _0} = 0,\;{\mu _0} = {\varepsilon ^{\frac{1}{{40\eta \left( {\tau  + 1} \right)}}}},\;{s_0} = \frac{{s{\gamma _0}}}{{16\left( {{M^ * } + 2} \right)K_1^{\tau  + 1}}}, \\
	&{\mathcal{V}_0} = \left\{\xi:\xi\in B_{\varrho_0}(\xi_0),\;\varrho_0= s_0^{\frac{1}{3}}\right\} \subset \mathcal{V},\;D \left( {{s_0},{r_0}} \right) = \left\{ {\left( {y,x} \right): |y| < {s_0},\left| {\operatorname{Im} x} \right| < {r_0}} \right\}	,\notag
\end{align}
where $ 0 < {s_0},{\gamma _0},{\mu _0} \leqslant 1$ are sufficiently small,  $\tau  >\max\{n-1,1\}  $, and $ {M^ * } > 0 $ is a constant defined as in Lemma \ref{FPKAM3.3}. Moreover,
	\[K_1=\left(\left[-\ln\mu_0\right]+1\right)^{3\eta}.\]
 Therefore, we can write with $(y,x,\xi)\in D(s_0,r_0)\times \mathcal{V}_0$ that 
\begin{align}
	{H_0}: ={}& H\left( {y,x,{\xi _0}} \right) = {N_0} + {P_0},\notag \\
	{N_0}: ={}& {N_0}\left( {y,{\xi _0},\varepsilon } \right) = {e_0} + \left\langle {\omega \left( {{\xi _0}} \right),y} \right\rangle  + {\bar h _0},\notag \\
	{P_0}: ={}& \varepsilon P\left( {y,x,{\xi _0} } \right).\notag
\end{align}
According to the above parameters, we have the following estimate for $ P_0 $.
\begin{lemma}%\label{FPKAM3.1}
	It holds that
	\begin{equation}\notag
		{\left\| {{P_0}} \right\|_{D\left( {{s_0},{r_0}} \right)}} \leqslant \gamma _0^{n+6}s_0^4{\mu _0}.
	\end{equation}
\end{lemma}
\begin{proof}
Utilizing $ \gamma_0^{n+6} = \varepsilon^{1/4} $ and $ \left[ -\ln \mu_0 \right] + 1 < 1/\mu_0 $, we obtain
\begin{align*}
	s_0^4 = \frac{s^4 \varepsilon^{\frac{4}{4(n+6)}}}{16^4 (M^* + 2)^4 K_1^{4(\tau + 1)}}
	> \frac{s^4 \varepsilon^{\frac{4}{4(n+6)}} \mu_0^{12\eta (\tau + 1)}}{16^4 (M^* + 2)^4} \geqslant \frac{s^4 \varepsilon^{\frac{1}{n+6} + \frac{3}{8}}}{16^4 (M^* + 2)^4}.
\end{align*}
At this point, we choose $ \varepsilon_0 > 0 $ sufficiently small such that
\begin{equation}\notag
	\varepsilon_0^{\frac{1}{8} - \frac{1}{40\eta(\tau+1)}} \left\| P_0 \right\|_{D(s_0, r_0)} \frac{16^4 (M^* + 2)^4}{s^4} \leqslant 1.
\end{equation}
By using $ \mu_0 = \varepsilon^{1/40\eta(\tau+1)} $, we have
\begin{equation}\label{FPKAMu0}
	\gamma_0^{n+6} s_0^4 \mu_0 \geqslant \frac{s^4 \varepsilon^{\frac{1}{n+6} + \frac{3}{8} + \frac{1}{4} + \frac{1}{40\eta(\tau+1)}}}{16^4 (M^* + 2)^4}
	\geqslant \frac{s^4 \varepsilon^{\frac{1}{4} + \frac{3}{8} + \frac{1}{4} + \frac{1}{40\eta(\tau+1)}}}{16^4 (M^* + 2)^4} = \varepsilon^{\frac{7}{8}} \frac{s^4 \varepsilon^{\frac{1}{40\eta(\tau+1)}}}{16^4 (M^* + 2)^4},
\end{equation}
and from \eqref{FPKAMu0} and $ 0 < \varepsilon < \varepsilon_0 $, we obtain
\[
\varepsilon^{\frac{1}{8} - \frac{1}{40\eta(\tau+1)}} \left\| P_0 \right\|_{D(s_0, r_0)} \frac{16^4 (M^* + 2)^4}{s^4} \leqslant 1,
\]
that is to say,
\begin{align}\label{FPKAMPP}
	\varepsilon^{\frac{1}{8}} \left\| P_0 \right\|_{D(s_0, r_0)} \leqslant \frac{s^4 \varepsilon^{\frac{1}{40\eta(\tau+1)}}}{16^4 (M^* + 2)^4}.
\end{align}
Finally, using \eqref{FPKAMu0} and \eqref{FPKAMPP}, we conclude that
\begin{align*}
	\left\| P_0 \right\|_{D(s_0, r_0)} = \varepsilon^{\frac{7}{8}} \varepsilon^{\frac{1}{8}} \left\| P_0 \right\|_{D(s_0, r_0)} \leqslant \varepsilon^{\frac{7}{8}} \frac{s^4 \varepsilon^{\frac{1}{40\eta(\tau+1)}}}{16^4 (M^* + 2)^4} \leqslant \gamma_0^{n+6} s_0^4 \mu_0.
\end{align*}
\end{proof}
\subsection{The induction from the $ \nu $-th KAM step}
In this section, we present all the KAM lemmas for the $ \nu $-th step. As previously mentioned, Lemma \ref{FPKAMcrucial} in Subsection \ref{kkkk} is the most critical. It utilizes assumptions \ref{aph1}, \ref{aph2}, and \ref{aph3} to ensure the persistence of the initial toral frequency $ \Upsilon = \omega(\xi_0) $.
\subsubsection{Description of the $ \nu $-th KAM step}
We now define the parameters of the $ \nu $-th KAM step, i.e.,
\begin{equation}\notag
	{r_\nu } = \frac{1}{2}{r_{\nu  - 1}} + \frac{1}{4}{r_0}, \quad {s_\nu } =\frac{1}{8} \mu _\nu ^{2\rho }{s_{\nu  - 1}}, \quad {\mu _\nu } = {8^4}\mu _{\nu  - 1}^{1 + \rho }.
\end{equation}
Suppose that at the $ \nu $-th step, we have arrived at the following real analytic Hamiltonian system
\begin{equation}\label{FPKAM1}
	{H_\nu } = {N_\nu } + {P_\nu }, \quad {N_\nu } = {e_\nu } + \left\langle {\omega \left( {{\xi _0}} \right),y} \right\rangle  + {\bar h _\nu }\left( {y,\xi } \right)
\end{equation}
defined on $ D\left( {{s_\nu },{r_\nu }} \right) $, with the perturbation estimate
\begin{equation}\label{equ11}
		{\left\| {{P_\nu }} \right\|_{D\left( {{s_\nu },{r_\nu }} \right)}} \leqslant \gamma _0^{n+6}s_\nu ^4{\mu _\nu } 
\end{equation}
The equations of the motion associated to $ H_\nu $ are
\begin{equation}\label{FPKAMmotion}
	\left\{ \begin{gathered}
		{\dot{y}_\nu } =  - {\partial _{{x_\nu }}}{H_\nu }, \hfill \\
		{\dot{x}_\nu } = {\partial _{{y_\nu }}}{H_\nu }. \hfill \\
	\end{gathered}  \right.
\end{equation}

Except for additional instructions, we omit the indices for all quantities of the present
KAM step (at the $ \nu $th-step) and use ``$ + $'' to index all the quantities (Hamiltonians, domains, normal forms, perturbations, transformations, etc.) in the next KAM step (at the $ (\nu +1)$-th step). To simplify the notations, we will not specify the dependence of $ P $, $ P_+ $ and etc. All the universal constants $ c_1$-$c_6 $ below are positive and independent of the iteration process, and we also use ``$ c $'' to denote any intermediate universal positive constant which is independent of the iteration process.

Define
\begin{align}
	{r_ + } ={}&\frac{1}{2} r + \frac{1}{4}{r_0},\notag \\
	{s_ + } ={}& \frac{1}{8}\alpha s, \quad \alpha  = {\mu ^{2\rho }} = {\mu ^{\frac{1}{5}}},\notag \\
\label{FPKAMCANSHU2}	{\mu _ + } ={}& {8^4}{c_0}{\mu ^{1 + \rho }}, \quad {c_0} = 1 + \mathop {\max }\limits_{1 \leqslant i \leqslant 6} {c_i}, \\
	{K_ + } ={}& {\left( {\left[ { - \ln \mu } \right] + 1} \right)^{3\eta }},\notag \\
	D\left( s \right) ={}& \left\{ {y \in {\mathbb{C}^n}:\left| y \right| < s} \right\},\notag \\
	\hat D ={}& D\left( {s,{r_ + } + \frac{7}{8}\left( {r - {r_ + }} \right)} \right),\notag \\
	\tilde D ={}& D\left( {\frac{1}{2}s,{r_ + } + \frac{3}{4}\left( {r - {r_ + }} \right)} \right),\notag \\
	{D_{\frac{i\alpha}{8}}} ={}& D\left( {\frac{i}{8}\alpha s,{r_ + } +  \frac{i-1}{8}\left( {r - {r_ + }} \right)} \right),\quad 1 \leqslant i \leqslant 8,\notag \\
	{D_ + } ={}& {D_{\frac{\alpha}{8}}} = D\left( {{s_ + },{r_ + }} \right),\notag \\
	\Gamma \left( {r - {r_ + }} \right) ={}& \sum\limits_{0 < \left| k \right| \leqslant {K_ + }} {{{\left| k \right|}^{3\tau  + 5}}{e^{ - \frac{1}{8}\left| k \right|\left( {r - {r_ + }} \right)}}}. \notag
\end{align}
\subsubsection{Construct a symplectic transformation}
In what follows, we construct a symplectic coordinate transformation $ \Phi _ +  $ as
\begin{equation}\notag
	{\Phi _ + }:\left( {{y_ + },{x_ + }} \right) \in D\left( {{s_ + },{r_ + }} \right) \to {\Phi _ + }\left( {{y_ + },{x_ + }} \right) = \left( {y,x} \right) \in D\left( {s,r} \right),
\end{equation}
such that it transforms the Hamiltonian \eqref{FPKAM1} into the Hamiltonian of the next KAM cycle (at the $ (\nu + 1) $-th step)
\begin{equation}\notag
	{H_ + } = H \circ {\Phi _ + } = {N_ + } + {P_ + },
\end{equation}
where $ N_+ $ and $ P_+ $ have the  similar properties as $ N $ and $ P $ respectively on $ D(s_+,r_+) $, and the equations of motion \eqref{FPKAMmotion} are changed into
\begin{equation}\notag
	\left\{ \begin{gathered}
		{y_ + } =  - {\partial _{{x_ + }}}{H_ + }, \hfill \\
		{x_ + } = {\partial _{{y_ + }}}{H_ + }. \hfill \\
	\end{gathered}  \right.
\end{equation}
\subsubsection{Truncation}
To deal with the small divisors in homological equations, an efficient approach is truncation. Consider the truncation $ R $ of the Taylor-Fourier series of the perturbation $ P $:
\[P = \sum\limits_{k \in {\mathbb{Z}^n},l \in \mathbb{Z}_ + ^n} {{p_{kl}}{y^l}{e^{\sqrt { - 1} \left\langle {k,x} \right\rangle }}} , \quad R = \sum\limits_{\left| k \right| \leqslant {K_ + },\left| l \right| \leqslant 4} {{p_{kl}}{y^l}{e^{\sqrt { - 1} \left\langle {k,x} \right\rangle }}} .\]
\begin{lemma}\label{FPKAM3.2}
	Assume that
	\begin{equation}\label{FPKAMH1}
		\int_{{K_ + }}^{ + \infty } {{t^n}{e^{ -\frac{1}{16} t\left( {r - {r_ + }} \right)}}dt}  \leqslant \mu .
	\end{equation}
	Then there exists a universal constant $ c_1>0 $ such that
	\begin{equation}\label{FPKAMR}
			{\left\| {P - R} \right\|_{{D_\alpha }}} \leqslant {c_1}\gamma _0^{n+6}{s^4}{\mu ^2}, \quad {\left\| R \right\|_{{D_\alpha }}} \leqslant {c_1}\gamma _0^{n+6}{s^4}\mu .
	\end{equation}
\end{lemma}
%\begin{remark}
%The proof is indeed the same as Lemma 3.2 in \cite{DL}. To be more precise,  for the H\"older semi-norm estimates in \cite{DL}, the denominator part $ |\xi-\zeta|^\beta $ remains unchanged, and only the numerator part is  analyzed. Therefore, we just need to replace $ |\xi-\zeta|^\beta $ by $ \left|\omega(\xi)-\omega(\zeta)\right| $ to obtain estimates for the $ \omega $-dependent semi-norm.
%\end{remark}
\begin{proof}
For convenience of expression, we denote
\[
\mathcal{I} = \sum_{|k| > K_+, l \in \mathbb{Z}_+^n} p_{kl} y^l e^{\sqrt{-1} \langle k, x \rangle}, \quad
\mathcal{II} = \sum_{|k| \leqslant K_+, |l| > 4} p_{kl} y^l e^{\sqrt{-1} \langle k, x \rangle}.
\]
Then, it is evident that
\begin{align*}
	P - R = \mathcal{I} + \mathcal{II}.
\end{align*}
We first estimate $\mathcal{I}$. Note that according to \eqref{equ11}, we have
\begin{align}\label{pki}
	\left| \sum_{l \in \mathbb{Z}_+^n} p_{kl} y^l \right| \leqslant |P|_{D(s,r)} e^{-|k| r} \leqslant \gamma_0^{n+6} s^4 \mu e^{-|k| r},
\end{align}
where the first inequality is widely used in KAM theory; see, for instance, the works of Chow et al. \cite{MR1938331},  Li and Yi \cite{MR1926285} and Salamon \cite{MR2111297}. Utilizing \eqref{FPKAMH1} and \eqref{pki}, we obtain
\begin{align}
	\left| \mathcal{I} \right|_{\hat{D}} &\leqslant \sum_{|k| > K_+} \left| \sum_{l \in \mathbb{Z}_+^n} p_{kl} y^l \right| e^{|k| \left( \frac{r_+}{8} + \frac{7r}{8} \right)} \notag \\
	&\leqslant \sum_{|k| > K_+} |P|_{D(s,r)} e^{-|k| \frac{r - r_+}{8}} \notag \\
	&\leqslant \gamma_0^{n+6} s^4 \mu \sum_{\kappa = K_+}^{\infty} \kappa^n e^{-\kappa \frac{r - r_+}{8}} \notag \\
	&\leqslant \gamma_0^{n+6} s^4 \mu \int_{K_+}^{\infty} t^n e^{-t \frac{r - r_+}{16}}  dt \notag \\
	&\leqslant \gamma_0^{n+6} s^4 \mu^2. \label{equ7}
\end{align}
From \eqref{equ11} and \eqref{equ7}, we obtain the estimate
\begin{align*}
	\left| P - \mathcal{I} \right|_{\hat{D}} \leqslant |P|_{D(s,r)} + |\mathcal{I}|_{\hat{D}} \leqslant 2 \gamma_0^{n+6} s^4 \mu.
\end{align*}
For $|\beta| = 5$, let $\int$ denote the $ |\beta| $-th order antiderivative of $\frac{\partial^{\beta}}{\partial y^{\beta}}$. By the Cauchy estimate of $P - \mathcal{I}$ in the domain $D_\alpha$, we have
\begin{align*}
	\left| \mathcal{II} \right|_{D_\alpha} &\leqslant \left| \int \frac{\partial^{\beta}}{\partial y^{\beta}} \sum_{|k| \leqslant K_+, |l| > 4} p_{kl} y^l e^{\sqrt{-1} \langle k, x \rangle} dy \right|_{D_\alpha} \\
	&\leqslant \left| \int \left| \frac{\partial^{\beta}}{\partial y^{\beta}} (P - \mathcal{I}) \right| dy \right|_{D_\alpha} \\
	&\leqslant \left| \frac{c}{s^5} \int \left| P - \mathcal{I} \right|_{\hat{D}} dy \right|_{D_\alpha} \\
&	\leqslant \frac{2c}{s^5} \gamma_0^{n+6} s^4 \mu {\left| {\int {dy} } \right|_{{D_\alpha }}}\\
	&\leqslant \frac{2c}{s^5} \gamma_0^{n+6} s^4 \mu (\alpha s)^5 \\
	&\leqslant c \gamma_0^{n+6} s^4 \mu^2.
\end{align*}
Therefore,
\begin{align}\label{equ12}
	\left| P - R \right|_{D_\alpha} = \left| \mathcal{I} + \mathcal{II} \right|_{D_\alpha} \leqslant c \gamma_0^{n+6} s^4 \mu^2,
\end{align}
and consequently,
\begin{align}\label{equ13}
	\left| R \right|_{D_\alpha} \leqslant |P - R|_{D_\alpha} + |P|_{D(s,r)} \leqslant c \gamma_0^{n+6} s^4 \mu.
\end{align}
Next, we need to estimate the norm $\left\| P - R \right\|_{\omega}$. For any $(y, x) \in D_\alpha$, we have
\begin{align}
	\left\| P - R \right\|_{\omega} &= \sup_{\xi \neq \zeta} \frac{\sup_{y, x \in G \times \mathbb{T}^n} \left| P(x, y, \xi) - R(x, y, \xi) - (P(x, y, \zeta) - R(x, y, \zeta)) \right|}{\left| \omega(\xi) - \omega(\zeta) \right|} \notag \\
	&\leqslant \sup_{\xi \neq \zeta} \frac{\sup_{y, x \in G \times \mathbb{T}^n} \left| \int \left| \frac{\partial^p}{\partial y^p} (P(x, y, \xi) - R(x, y, \xi) - (P(x, y, \zeta) - R(x, y, \zeta))) \right| dy \right|}{\left| \omega(\xi) - \omega(\zeta) \right|} \notag \\
	&\leqslant \sup_{\xi \neq \zeta} \frac{\sup_{y, x \in G \times \mathbb{T}^n} \left| \frac{c}{s^6} \int \left| P(x, y, \xi) - P(x, y, \zeta) \right| dy \right|}{\left| \omega(\xi) - \omega(\zeta) \right|} \notag \\
	\label{FPKAMDISANGE}		&\leqslant \sup_{\xi \neq \zeta} \frac{c}{s^5} \frac{\sup_{y, x \in G \times \mathbb{T}^n} \left| P(x, y, \xi) - P(x, y, \zeta) \right|}{\left| \omega(\xi) - \omega(\zeta) \right|} (\alpha s)^5  \\
	&\leqslant c \mu \|P\|_{\omega} \notag \\
	&\leqslant c \gamma_0^{n+6} s^4 \mu^2. \label{equ14}
\end{align}
Here, \eqref{FPKAMDISANGE} uses the Cauchy estimate, and \eqref{equ14} uses \eqref{equ11}. Similarly, we obtain
\begin{align}\label{equ15}
	\|R\|_{\omega} \leqslant \|P - R\|_{\omega} + \|P\|_{\omega} \leqslant c \gamma_0^{n+6} s^4 \mu.
\end{align}
Under these grounds, it follows from \eqref{equ12}, \eqref{equ13}, \eqref{equ14}, and \eqref{equ15} that the estimate in \eqref{FPKAMR} holds, which completes the proof of Lemma \ref{FPKAM3.2}.
\end{proof}
\subsubsection{The quasi-linear homological equations}
We construct a symplectic transformation as the time $ 1 $-mapping $ \phi _F^1 $ of the flow generated by a Hamiltonian $ F $ to eliminate all resonance terms in $ R $. By setting
\begin{equation}\label{FPKAMFXS}
	F = \sum\limits_{0 < \left| k \right| \leqslant {K_ + },\left| l \right| \leqslant 4} {{f_{kl}}{y^l}{e^{\sqrt { - 1} \left\langle {k,x} \right\rangle }}} , \quad [R] = {\left( {2\pi } \right)^{ - n}}\int_{{\mathbb{T}^n}} {R\left( {y,x} \right)dx} ,
\end{equation}
we obtain that
\[	\left\{ {N,F} \right\} + R - \left[ R \right] = 0,\]
which yields the so-called quasi-linear homological equations
\begin{equation}\label{FPKAMtdfc}
	\sqrt { - 1} \left\langle {k,\omega \left( {{\xi _0}} \right) + {\partial _y}\bar h } \right\rangle {f_{kl}} = {p_{kl}}, \quad \left| l \right| \leqslant 4,\\
	\;0 < \left| k \right| \leqslant {K_ + }.
\end{equation}
\begin{lemma}\label{FPKAM3.3}
	Assume that
	\begin{equation}\label{FPKAMYINLI3}
		\mathop {\max }\limits_{\left| i \right| \leqslant 2} {\left\| {\partial _y^i\bar h  - \partial _y^i{{\bar h }_0}} \right\|_{D\left( s \right)}} \leqslant \mu _0^{\frac{1}{2}}, \quad s < \frac{{{\gamma _0}}}{{ {6\left( {{M^ * } + 2} \right)K_ + ^{\tau  + 1}} }},
	\end{equation}
	where
	\[{M^ * } = \mathop {\max }\limits_{\left| i \right| \leqslant 2,y \in D\left( s \right)} \left| {\partial _y^i{{\bar h }_0}\left( {{\xi _0},y} \right)} \right|.\]
	Then the quasi-linear homological equations in \eqref{FPKAMtdfc} can be uniquely solved on $ D(s) $ to obtain a family of functions $ f_{kl} $ which are analytic in $ y $ and satisfy
	\begin{equation}\notag
		{\left\| {\partial _y^i{f_{kl}}} \right\|_{D\left( s \right)}} \leqslant {c_2}{\left| k \right|^{2\left( {\left| i \right| + 1} \right)\tau  +2 \left| i \right|+1}}\gamma _0^{n+4 - 2\left| i \right|}{s^{4 - \left| l \right|}}\mu {e^{ - \left| k \right|r}}
	\end{equation}
	for $ \left| l \right| \leqslant 4,0 < \left| k \right| \leqslant {K_ + }$ and $\left| i \right| \leqslant 2 $,	where $ c_2>0 $ is a universal constant.
\end{lemma}
\begin{proof}
For any $y \in D(s)$, it follows from \eqref{FPKAMYINLI3} that
\begin{equation*}
	\left| \partial_y \bar{h} \right|_{D(s)} = \left| (\partial_y \bar{h} - \partial_y \bar{h}_0) + \partial_y \bar{h}_0 \right|_{D(s)}
	\leqslant (1 + M^*) |y| < (1 + M^*) s
	< \frac{\gamma_0}{6 |k|^{\tau + 1}},
\end{equation*}
and
\begin{equation*}
	\left[ \partial_y \bar{h} \right]_{\omega} = \sup_{\xi \neq \zeta} \frac{\left| \partial_y \bar{h}(y, \xi) - \partial_y \bar{h}(y, \zeta) \right|}{\left| \omega(\xi) - \omega(\zeta) \right|} \leqslant \mu_0^{\frac{1}{2}} |y| < |s| < \frac{\gamma_0}{6 |k|^{\tau + 1}},
\end{equation*}
which implies that
\begin{equation}\label{eqb36}
	\left\| \partial_y \bar{h} \right\|_{D(s)} < \frac{\gamma_0}{3 |k|^{\tau + 1}}.
\end{equation}
Consequently, it follows from \eqref{eqb36} and the Diophantine nonresonance \eqref{FPKAMDIO} in the Internal Condition \ref{aph1} that for $\xi \in \mathcal{V}$, $y \in D(s)$, and $0 < |k| \leqslant K_+$, we have
\begin{align}\label{eqb37}
	\left| \left\langle k, \omega(\xi) + \partial_y \bar{h}(y) \right\rangle \right| &\geqslant \left| \left\langle k, \omega(\xi_\nu) \right\rangle \right| - \left| \left\langle k, \omega(\xi) - \omega(\xi_\nu) \right\rangle \right| - \left| \left\langle k, \partial_y \bar{h}(y) \right\rangle \right| \notag \\
	&> \frac{\gamma_0}{|k|^\tau} - \frac{\gamma_0}{3 |k|^\tau} - \frac{\gamma_0}{3 |k|^\tau} = \frac{\gamma_0}{3 |k|^\tau}.
\end{align}
Therefore, the operator
\begin{equation}\label{eqb38}
	L_k = \sqrt{-1} \left\langle k, \omega(\xi) + \partial_y \bar{h}(y) \right\rangle
\end{equation}
is invertible, and for any $y \in D(s)$, $\xi \in \mathcal{V}$, $0 < |k| \leqslant K_+$, and $|l| \leqslant 4$, we have
\begin{equation}\label{fki}
	f_{kl} = L_k^{-1} p_{kl}.
\end{equation}
Note that the analyticity of $\bar{h}(y)$ and $p_{kl}$ implies that the coefficients $f_{kl}$ are also analytic in the action variable $y$.
Let $0 < |k| \leqslant K_+$. By the first inequality in \eqref{pki} and the Cauchy estimate, we have
\begin{align}\label{eqb9}
	\left\| p_{kl} \right\|_{D(s)} &\leqslant \left\| \partial_y^l P \right\|_{\tilde{D}} e^{-|k| r} \leqslant \gamma_0^{n+6} s^{4 - |l|} \mu e^{-|k| r},  \quad  |l| \leqslant 4,
\end{align}
and according to \eqref{eqb37} and \eqref{eqb38}, we obtain
\begin{align*}
	\left\| L_k^{-1} \right\|_{D(s)} &\leqslant \frac{c |k|^{2\tau + 1}}{\gamma_0^2},
\end{align*}
and
\begin{align}\label{lk-1}
	\left\| \partial_y^i L_k^{-1} \right\|_{D(s)} &\leqslant c_2 |k|^{|i|} \left\| L_k^{-1} \right\|_{D(s)}^{|i| + 1} \leqslant \frac{c_2 |k|^{2(|i| + 1)\tau + 2|i| + 1}}{\gamma_0^{2(|i| + 1)}},  \quad  |i| \leqslant 2.
\end{align}
Therefore, combining \eqref{fki}, \eqref{eqb9}, and \eqref{lk-1}, we finally obtain
\begin{align*}
	\left\| \partial_y^i f_{kl} \right\|_{D(s)} &\leqslant \frac{c_2 |k|^{2(|i| + 1)\tau + 2|i| + 1}}{\gamma_0^{2(|i| + 1)}} \gamma_0^{n+6} s^{4 - |l|} \mu e^{-|k| r} \\
	&= c_2 |k|^{2(|i| + 1)\tau + 2|i| + 1} \gamma_0^{n + 4 - 2|i|} s^{4 - |l|} \mu e^{-|k| r},  \quad  |i| \leqslant 2.
\end{align*}
This completes the proof of Lemma \ref{FPKAM3.3}.
\end{proof}

Applying the previously constructed coordinate transformation $ \phi_F^1 $ to the Hamiltonian $ H $, we obtain
\begin{align*}
	H \circ \phi_F^1 =&\; (N + R) \circ \phi_F^1 + (P - R) \circ \phi_F^1 \\
	= &\;(N + R) + \{ N, F \} + \int_0^1 \left\{ (1 - t) \{ N, F \} + R, F \right\} \circ \phi_F^t \,dt \\
	\quad&\; + (P - R) \circ \phi_F^1 \\
=&\; N + [R] + \int_0^1 \left\{ R_t, F \right\} \circ \phi_F^t \, dt + (P - R) \circ \phi_F^1 \\
	:=& \; N_+ + P_+,
\end{align*}
where the quantities are defined as follows:
\begin{align}
	N_+ &= N + [R] = e_+ + \langle \omega_+(\xi), y \rangle + \bar{h}_+(y, \xi), \label{equ5} \\
	e_+ &= e + p_{00}^\nu, \label{equ8} \\
	\omega_+ &= \omega_0(\xi) + \sum_{j=0}^{\nu} p_{01}^j(\xi), \notag \\
	\bar{h}_+ &= \bar{h}(y, \xi) + [R] - p_{00}^\nu - \left\langle p_{01}^\nu(\xi), y \right\rangle, \label{eqb11} \\
	P_+ &= \int_0^1 \{ R_t, F \} \circ \phi_F^t \;dt + (P - R) \circ \phi_F^1, \notag \\
	R_t &= (1 - t) [R] + t R. \notag
\end{align}
\subsubsection{The  parameter translation}
In this section, we construct a translation $ \phi $ so as to keep the frequency unchanged, where
\[\phi : \quad x \to x, \quad y \to y, \quad \widetilde \xi  \to \widetilde \xi  + {\xi _ + } - \xi ,\]
and $ {\xi _ + } $ will be determined later. Let $ {\Phi _ + } = \phi _F^1 \circ \phi  $. Then
\begin{align*}
	{}&H \circ {\Phi _ + } = {N_ + } + {P_ + },\\
	{}&{N_ + } = {\bar N _ + } \circ \phi  = {e_ + } + \left\langle {\omega \left( {{\xi _ + }} \right),y} \right\rangle  + \left\langle {\sum\limits_{j = 0}^\nu  {p_{01}^j\left( {{\xi _ + }} \right)} ,y} \right\rangle  + {\bar h _ + }\left( {y,{\xi _ + }} \right),\\
	{}&{P_ + } = {\bar P _ + } \circ \phi .
\end{align*}

\subsubsection{Frequency-preserving
}\label{kkkk}
In this section, we  show that the prescribed Diophantine frequency can be preserved by a suitable parameter translation technique in the iteration process. To be more precise, the Internal Condition \ref{aph1} ensures that the parameter $ {{\xi _ \nu }} $ can be found in the internal parameter set $ \mathcal{O}^o $ to keep the frequency unchanged at this KAM step, and the  Relative Singularity Condition \ref{aph2} together with the Controllability Condition \ref{aph3} assure that $ \left\{ {{\xi _\nu }} \right\}_{\nu \in \mathbb{N}^+} $ is indeed a Cauchy sequence. The following lemma is crucial to our arguments.
\begin{lemma}\label{FPKAMcrucial}
	Assume that
	\begin{equation}\label{FPKAMYINLI4JIASHE}
		{\mathop {\sup }\limits_{\xi  \in \mathcal{O}} \left| {\sum\limits_{j = 0}^\nu  {p_{01}^j} } \right|} < \mu _0^{\frac{1}{2}}.	
	\end{equation}
	Then there exists at least a $ {\xi _ + } \in \mathcal{V} \subset {\mathcal{O}^o} $ such that
	\begin{equation}\label{FPKAMomega}
		\omega \left( {{\xi _ + }} \right) + \sum\limits_{j = 0}^\nu  {p_{01}^j\left( {{\xi _ + }} \right)}  = \omega \left( {{\xi _0}} \right).
	\end{equation}
Moreover, $ {\lim _{\nu  \to  + \infty }}{\xi _ + } : = {\xi ^ * } \in \mathcal{V} $, and $ \left| {{\xi ^ * } - {\xi _0}} \right| =  \mathcal{O}\left( \varphi \left( \varepsilon  \right){ - \int_0^{\varepsilon } {\frac{{\varphi (x)}}{{x\ln x}}dx} } \right) = o\left( 1 \right) $, where the function $ \varphi $ is defined in the Controllability Condition \ref{aph3}.
\end{lemma}
\begin{proof}
	The proof will be completed by an induction on $ \nu $. We start with the case $ \nu=0 $. It is evident that $ \omega \left( {{\xi _0}} \right) = \omega \left( {{\xi _0}} \right) $. Now, assume that for some $ \nu  \geqslant 1 $, we have arrived at
	\begin{equation}\notag 
		\omega \left( {{\xi _i}} \right) + \sum\limits_{j = 0}^{i - 1} {p_{01}^j\left( {{\xi _i}} \right)}  = \omega \left( {{\xi _0}} \right), \quad {\xi _i} \in   \mathcal{V} \subset \mathcal{O}^o, \quad 1 \leqslant i \leqslant \nu .
	\end{equation}
Then, in view of the smallness of $ \sum\nolimits_{j = 0}^{i - 1} {p_{01}^j\left( {\cdot} \right)} $ and the Internal Condition \ref{aph1}, we can conclude that there exists at least a $ {\xi _ + } \subset \mathcal{O}^o $, such that the frequency-preserving equation  \eqref{FPKAMomega} holds, whenever  $ \varepsilon>0 $ is sufficiently small. Indeed, we could show that $ \xi_+ \in \mathcal{V} $ by the the following stronger convergence analysis, as detailed below.

	Recall that \eqref{FPKAMR} in Lemma \ref{FPKAM3.2} implies that
	\begin{equation}\label{FPKAMCANSHU4}
			{\left|{p_{01}^j}\right|,\left[{p_{01}^j} \right]_{\omega}} < c{\mu _j}, \quad 0 \leqslant j \leqslant \nu,
	\end{equation}
which leads to 
	\begin{equation}\label{FPKAMlxm}
		\left| {p_{01}^j\left( {{\xi _ + }} \right) - p_{01}^j\left( \xi  \right)} \right| < c{\mu _j}\left|\omega(\xi_+)-\omega(\xi)\right|, \quad 0 \leqslant j \leqslant \nu.
	\end{equation}
Here, we first provide a  quantitative estimate for $ \mu_j $, showing that it is super-exponentially small. Using \eqref{FPKAMCANSHU} and \eqref{FPKAMCANSHU2}, we obtain from the smallness of $ \varepsilon>0 $ that
\begin{align}
{\mu _j} &\leqslant C\mu _{j - 1}^{1 + \rho } \leqslant {C^{1 + \left( {1 + \rho } \right)}}\mu _{j - 2}^{{{\left( {1 + \rho } \right)}^2}} \leqslant  \cdots  \leqslant {C^{\sum\limits_{l = 0}^{j - 1} {{{\left( {1 + \rho } \right)}^l}} }}\mu _0^{{{\left( {1 + \rho } \right)}^j}}\notag \\
\label{FPKAMCANSHU3}& = {C^{\frac{{{{\left( {1 + \rho } \right)}^{j - 1}} - 1}}{\rho }}}{\varepsilon ^{\frac{{{{\left( {1 + \rho } \right)}^j}}}{{40\eta \left( {\tau  + 1} \right)}}}} \leqslant \frac{{{\varepsilon ^{{q^j}}}}}{2c}, \quad \forall j \in {\mathbb{N}^ + },
\end{align} 
where $ 1<q<1+\rho $, and $ C>0 $ is a universal constant independent of $ j$ and $ \varepsilon $. Besides, with the monotonicity, we have
\[\int_{{q^{j - 1}}}^{{q^j}} {\frac{{{\varepsilon ^x}}}{x}dx}  \geqslant \int_{{q^{j - 1}}}^{{q^j}} {\frac{{{\varepsilon ^{{q^j}}}}}{x}dx}  \geqslant \int_{{q^{j - 1}}}^{{q^j}} {\frac{{{\varepsilon ^{{q^j}}}}}{{{q^j}}}dx}  = \left( {{q^j} - {q^{j - 1}}} \right)\frac{{{\varepsilon ^{{q^j}}}}}{{{q^j}}}, \quad j \in \mathbb{N},\]
and this leads to 
\begin{align}
\sum\limits_{j = 0}^\infty  {{\mu _j}} & \leqslant \sum\limits_{j = 0}^\infty  {\frac{{{\varepsilon ^{{q^j}}}}}{{2c}}}  = \frac{1}{{2c}}\left( {\varepsilon  + \frac{q}{{q - 1}}\sum\limits_{j = 1}^\infty  {\left( {{q^j} - {q^{j - 1}}} \right)\frac{{{\varepsilon ^{{q^j}}}}}{{{q^j}}}} } \right)\notag \\
& \leqslant \frac{1}{{2c}}\left( {\varepsilon  + \frac{q}{{q - 1}}\sum\limits_{j = 1}^\infty  {\int_{{q^{j - 1}}}^{{q^j}} {\frac{{{\varepsilon ^x}}}{x}dx} } } \right) = \frac{1}{{2c}}\left( {\varepsilon  + \frac{q}{{q - 1}}\int_1^{ + \infty } {\frac{{{\varepsilon ^x}}}{x}dx} } \right)\notag \\
\label{FPKAMLEIJIA}& = \frac{1}{{2c}}\left( {\varepsilon  + \frac{q}{{q - 1}}\int_0^\varepsilon  {\frac{1}{{ - \ln x}}dx} } \right) \leqslant \frac{1}{{2c}}\left( {\varepsilon  + \frac{q}{{q - 1}}\left( {\frac{\varepsilon }{{ - \ln \varepsilon }}} \right)} \right) = o\left( 1 \right)
\end{align}
as $ 0<\varepsilon \ll 1 $. Now, with the Relative Singularity Condition \ref{aph2}, we have
\begin{align}
	\left| {p_{01}^\nu \left( {{\xi _ + }} \right)} \right| ={}& \left| {\omega \left( {{\xi _ + }} \right) - \omega \left( \xi  \right) + \sum\limits_{j = 0}^{\nu  - 1} {\left( {p_{01}^j\left( {{\xi _ + }} \right) - p_{01}^j\left( \xi  \right)} \right)} } \right|\notag \\
	\geqslant{}& \left| {\omega \left( {{\xi _ + }} \right) - \omega \left( \xi  \right)} \right| - \sum\limits_{j = 0}^{\nu  - 1} {\left| {p_{01}^j\left( {{\xi _ + }} \right) - p_{01}^j\left( \xi  \right)} \right|} \notag \\
	\label{FPKAM1112}	\geqslant{}& \left| {\omega \left( {{\xi _ + }} \right) - \omega \left( \xi  \right)} \right| - c\left( {\sum\limits_{j = 0}^{\nu  - 1} {{\mu _j}} } \right)\left| {\omega \left( {{\xi _ + }} \right) - \omega \left( \xi  \right)} \right| \\
	\label{FPKAM1111}	\geqslant{}& \frac{1}{2}\left| {\omega \left( {{\xi _ + }} \right) - \omega \left( \xi  \right)} \right|,
\end{align}
where \eqref{FPKAM1112} uses the smallness of the $ \omega $-dependent semi-norm in \eqref{FPKAMlxm}, and \eqref{FPKAM1111} uses the smallness of the accumulated KAM errors in \eqref{FPKAMLEIJIA}, whenever $ \varepsilon>0  $ is sufficiently small. Then, by \eqref{FPKAMCANSHU4}, \eqref{FPKAMCANSHU3} and the Controllability Condition \ref{aph3}, one can derive from \eqref{FPKAM1111} that 
\begin{equation}\label{FPKAMCANSHU5}
	\left| {{\xi _{+}} - {\xi }} \right| \leqslant \varphi (\varepsilon^{{q^\nu }}).
\end{equation}
Then, in a manner analogous to the argument presented in \eqref{FPKAMLEIJIA}, we can demonstrate that
\begin{align}
\sum\limits_{\nu  = 0}^\infty  {\varphi ({\varepsilon ^{{q^\nu }}})} & = \varphi \left( \varepsilon  \right) + \frac{q}{{q - 1}}\sum\limits_{\nu  = 1}^\infty  {\left( {{q^\nu } - {q^{\nu  - 1}}} \right)\frac{{\varphi ({\varepsilon ^{{q^\nu }}})}}{{{q^\nu }}}} \notag \\
& \leqslant \varphi \left( \varepsilon  \right) + \frac{q}{{q - 1}}\sum\limits_{\nu  = 1}^\infty  {\int_{{q^{\nu  - 1}}}^{{q^\nu }} {\frac{{\varphi ({\varepsilon ^x})}}{x}dx} } \notag \\
& = \varphi \left( \varepsilon  \right) + \frac{q}{{q - 1}}\int_1^{ + \infty } {\frac{{\varphi ({\varepsilon ^x})}}{x}dx} \notag \\
\label{FPKAMCANSHU6}& = \varphi \left( \varepsilon  \right) - \frac{q}{{q - 1}}\int_0^\varepsilon  {\frac{{\varphi \left( x \right)}}{{ x\ln x}}dx}  = o\left( 1 \right)
\end{align}
as $ 0<\varepsilon \ll 1 $,   thanks to the monotonicity, continuity and integrability of $ \varphi $ in the Controllability Condition \ref{aph3}. Therefore, we eventually deduce from  Lebesgue's dominated convergence theorem that $ \{\xi_\nu\}_{\nu\in \mathbb{N}^+} $ is indeed a Cauchy sequence, hence it converges to some parametric limit $ \xi_\infty:= \xi^* $  in the subset $ \mathcal{V}\subset \mathcal{O}^o$. Furthermore, an upper bound of the distance between $ \xi^* $ and $ \xi_0 $ could be obtained from \eqref{FPKAMCANSHU5} and \eqref{FPKAMCANSHU6} as
\begin{align*}
	\left| {{\xi ^ * } - {\xi _0}} \right| &= 	\mathop {\lim }\limits_{\nu  \to  + \infty } \left| {{\xi _ {\nu+1} } - {\xi _0}} \right|=\mathop {\lim }\limits_{\nu  \to  + \infty } \left| {\sum\limits_{j = 0}^\nu  {\left({\xi _j} - {\xi _{j + 1}}\right)} } \right| \notag \\
	&\leqslant\mathop {\overline {\lim } }\limits_{\nu  \to  + \infty } \sum\limits_{j = 1}^{\nu  + 1} {\left| {{\xi _j} - {\xi _{j - 1}}} \right|}  \leqslant \sum\limits_{\nu  = 0}^\infty  {\varphi (\varepsilon ^{{q^\nu }})}\notag \\
	&=\mathcal{O}\left(\varphi \left( \varepsilon  \right) { - \int_0^{\varepsilon } {\frac{{\varphi (x)}}{{x\ln x}}dx} } \right) = o\left( 1 \right), \quad \varepsilon \to 0^+ ,
\end{align*}
  which gives the conclusion. 
\end{proof}

We recall here the topological degree argument of \cite{DL}. In spirit, it is identical to our analysis based on \ref{aph1}; replacing \ref{aph1} by the more explicit  topological degree condition in the finite-dimensional setting would not alter the proof of Lemma 3.4 in any way.   The true novelty lies in the analysis built upon  \ref{aph2} and \ref{aph3}. As emphasized in the Introduction, these assumptions allow the frequency mapping and the perturbation of the Hamiltonian system to possess explicitly arbitrarily weak regularity or explicitly very good regularity, thereby providing a generic approach to frequency-preserving KAM.

	\subsubsection{The estimate on $ {N_ + } $}
\begin{lemma}\label{FPKAMYINLI5}
	There is a universal constant $ c_3>0 $ such that the following hold:
	\begin{equation}\label{FPKAMYINLI53}
		\left| {{\xi _ + } - \xi } \right| \leqslant {c_3}\mu , \quad \left| {{e_ + } - e} \right| \leqslant {c_3}{s^4}\mu , \quad {\left\| {{{\bar h }_ + } - \bar h } \right\|_{D\left( s \right)}} \leqslant {c_3}{s^4}\mu .
	\end{equation}
\end{lemma}
\begin{proof}
The first estimate is already implied in the proof of Lemma \ref{FPKAMcrucial}. The other two inequalities can be directly obtained from \eqref{equ5} and \eqref{equ8}.
\end{proof}

\subsubsection{The estimate on $ {\Phi _ + } $}
\begin{lemma}\label{FPKAMYINLI6}
	There is a universal constant $ c_4>0 $ such that for all $ \left| i \right| + \left| j \right| \leqslant 4 $,
	\begin{equation}\notag
		{\left\| {\partial _x^i\partial _y^jF} \right\|_{\hat D}} \leqslant {c_4}\gamma _0^{n+4-2 \left| i \right|}{s^{4 - \left| i \right|}}\mu \Gamma \left( {r - {r_ + }} \right).
	\end{equation}
\end{lemma}
\begin{proof}
	By \eqref{FPKAMFXS} and Lemma \ref{FPKAM3.3}, we can prove that
	\begin{align*}
		\left\| \partial_x^j\partial_y^iF\right\| _{\hat D}&\leqslant \sum_{\vert l \vert\leqslant 4,0<\vert k\vert\leqslant K_+}\vert k\vert^j \left\|\partial_y^i(f_{kl}y^l)\right\| _{D(s)}e^{\vert k\vert\left(r_++\frac{7}{8}(r-r_+)\right)}\\
		&\leqslant c_4\sum_{0<\vert k\vert\leqslant K_+}\vert k\vert^{2(|i|+1)\tau+2\vert i\vert+1+\vert j\vert}\gamma_0^{n+4-2\vert i\vert}s^{4-\vert i\vert}\mu e^{-\vert k\vert\frac{r-r_+}{8}}\\
		&\leqslant c_4\gamma_0^{n+4-2\vert i\vert}s^{4-\vert i\vert}\mu\Gamma(r-r_+).
	\end{align*}
\end{proof}

\begin{lemma}\label{FPKAMYINLI7}
	Assume that
\begin{equation}\label{FPKAMYINLI71}
	{c_4}{s^{ 3}}\mu \Gamma \left( {r - {r_ + }} \right) < \frac{1}{8}\left( {r - {r_ + }} \right), \quad {c_4}{s^4}\mu \Gamma \left( {r - {r_ + }} \right) <\frac{1}{8} \alpha s.
\end{equation}
	Then for all $ 0 \leqslant t \leqslant 1 $, the mapping $ \phi _F^t:{D_{\alpha /4}} \to {D_{\alpha /2}} $ is well defined, and $ {\Phi _ + }:{D_ + } \to D\left( {s,r} \right) $. Moreover, there is a universal constant $ c_5>0 $ such that
	\begin{align}
		{\left\| {\phi _F^t - \rm{id}} \right\|_{\tilde D}},\;{\left\| {D\phi _F^t - \rm{Id}} \right\|_{\tilde D}},\;{\left\| {{D^2}\phi _F^t} \right\|_{\tilde D}},&\notag \\
	\label{IIIKKK}	{\left\| {{\Phi _ + } - \rm{id}} \right\|_{\tilde D}},\;{\left\| {D{\Phi _ + } - \rm{Id}} \right\|_{\tilde D}},\;{\left\| {{D^2}{\Phi _ + }} \right\|_{\tilde D}}& \leqslant {c_5}\mu \Gamma \left( {r - {r_ + }} \right).
	\end{align}
\end{lemma}
\begin{proof}
	Let $\phi_{F_1}^t$ and $\phi_{F_2}^t$ denote the components of the map $\phi_{F}^t$ in the $y$-plane and $x$-plane, respectively. Let $X_F = (F_y, -F_x)^\top$ be the vector field generated by $F$. We have
	\begin{align}\label{eqb19}
		\phi_F^t = \mathrm{id} + \int_0^t X_F \circ \phi_F^u \, du,\quad 0 \leqslant t \leqslant 1.
	\end{align}
	%For any $(y, x) \in D_{\alpha/4}$, let $t_* = \sup\{t \in [0,1]: \phi_F^t(y,x) \in D_\alpha\}$. 
	For any $0 \leqslant t \leqslant 1$, by $(y,x) \in D_{\alpha/4}$, Lemma \ref{FPKAMYINLI6} (applied with $ |i|=1, |j|=0 $ and $ |i|=0, |j|=1 $), and assumption \eqref{FPKAMYINLI71}, we can obtain the following estimates:
	\begin{align*}
\left| \phi_{F_1}^t(y,x)-y\right|\leqslant   \int_0^t \left\|F_x \circ \phi_F^u\right\|_{D_{\frac{\alpha}{4}}} du \leqslant  c_4 s^{4} \mu \Gamma(r - r_+) < \frac{1}{8}\alpha s,
	\end{align*}
	and
	\begin{align*}
\left|\phi_{F_2}^t(y,x)-x\right| \leqslant  \int_0^t \left\|F_y \circ \phi_F^u\right\|_{D_{\frac{\alpha}{4}}} du  \leqslant c_4 s^{3} \mu \Gamma(r - r_+)  <\frac{1}{8}(r - r_+).
	\end{align*}
	Therefore, 	by
\[D\left( {\frac{1}{4}\alpha s + \frac{1}{8}\alpha s,\left( {{r_ + } + \frac{1}{8}\left( {r - {r_ + }} \right)} \right) + \frac{1}{8}\left( {r - {r_ + }} \right)} \right) \subset D\left( {\frac{1}{2}\alpha s,{r_ + } + \frac{3}{8}\left( {r - {r_ + }} \right)} \right) = {D_{\frac{\alpha }{2}}},\]
we have $\phi_F^t(y,x) \in D_{\alpha/2} $. %$\phi_F^t \in D_{\alpha/2} \subset D_\alpha$, which implies {\color{blue}{$t_* = 1$}}. 
Moreover, $ {\Phi _ + }:{D_ + } \to D\left( {s,r} \right) $ can be directly derived.
	
	Next, we consider proving the first there estimates  in \eqref{IIIKKK}. By Lemma \ref{FPKAMYINLI6} and \eqref{eqb19}, we have
	\begin{align*}
		\left\|\phi_F^t - \mathrm{id}\right\|_{\tilde{D}} \leqslant c_5 \mu \Gamma(r - r_+).
	\end{align*}
	With Lemma \ref{FPKAMYINLI6}, \eqref{eqb19}, and the Gronwall inequality, we obtain
	\begin{align*}
		\left\|D\phi_F^t - \mathrm{Id}\right\|_{\tilde{D}} &\leqslant \left\|\int_0^t D X_F \circ \phi_F^\lambda D \phi_F^\lambda \, d\lambda\right\|_{\tilde{D}} \\
		&\leqslant \int_0^t \left\|D X_F \circ \phi_F^\lambda\right\|_{\tilde{D}} \left\|D \phi_F^\lambda - \mathrm{Id}\right\|_{\tilde{D}}d\lambda + \int_0^t \left\|D X_F \circ \phi_F^\lambda\right\|_{\tilde{D}}  d\lambda \\
		&\leqslant c_5 \mu \Gamma(r - r_+).
	\end{align*}
	Utilizing induction and similar arguments, we can obtain the following estimate for the second derivative of $\phi_F^t$:
	\begin{align*}
		\left\|D^2 \phi_F^t\right\|_{\tilde{D}} \leqslant c_5 \mu \Gamma(r - r_+).
	\end{align*}
 Moreover, the remaining three estimates in \eqref{IIIKKK} can be directly derived from the above estimates. Thus, we complete the proof of Lemma \ref{FPKAMYINLI7}.
\end{proof}

	\subsubsection{The estimate on $ {P _ + } $}
\begin{lemma}\label{FPKAMYINLI8}
	Assume that all the assumptions in the previous lemmas hold. Then there is a universal constant $ c_6>0 $ such that
	\begin{equation}\notag
		{\left\| {{P_ + }} \right\|_{{D_ + }}} \leqslant {c_6}\gamma _0^{n+6}{s^4}{\mu ^2}\left( {{\Gamma ^2}\left( {r - {r_ + }} \right) + \Gamma \left( {r - {r_ + }} \right)} \right).
	\end{equation}
	Moreover, if
	\begin{equation}\label{FPKAMYINLI8JIASHE}
		{\mu ^\rho }\left( {{\Gamma ^2}\left( {r - {r_ + }} \right) + \Gamma \left( {r - {r_ + }} \right)} \right) \leqslant 1,
	\end{equation}
	then
	\begin{equation}\label{FPKAMyinli83}
		{\left\| {{P_ + }} \right\|_{{D_ + }}} \leqslant {c_6}\gamma _0^{n+6}s_ + ^4{\mu _ + }.
	\end{equation}
\end{lemma}
\begin{proof}
	By \eqref{FPKAMH1} in Lemma \ref{FPKAM3.2}, Lemma \ref{FPKAMYINLI6} and Lemma \ref{FPKAMYINLI7}, we have that for all $0 \leqslant t \leqslant 1$,
	\[
	\left\| \{R_t,F\} \circ \phi_F^t \right\|_{D_{\frac{\alpha}{4}}} \leqslant c \gamma_0^{n+6} s^4 \mu^2 \Gamma^2(r - r_+),
	\]
	and
	\[
	\left\| (P - R) \circ \phi_F^1 \right\|_{D_{\frac{\alpha}{4}}} \leqslant c \gamma_0^{n+6} s^4 \mu^2 \Gamma(r - r_+).
	\]
	Therefore, by \eqref{eqb11}, we obtain
	\[
	\left\| P_+ \right\|_{D_{\frac{\alpha}{4}}} \leqslant c \gamma_0^{n+6} s^4 \mu^2 \left( \Gamma^2(r - r_+) + \Gamma(r - r_+) \right).
	\]
	Under the assumption \eqref{FPKAMYINLI8JIASHE}, we have
	\begin{align*}
		\left\| P_+ \right\|_{D_+} &\leqslant 8^4 c_0 \mu^{1+\rho} s_+^4 \mu^{1-2\rho-\frac{4}{5}} \gamma_0^{n+6} \left( \mu^\rho \left( \Gamma^2(r - r_+) + \Gamma(r - r_+) \right) \right) \\
		&\leqslant \gamma_0^{n+6} s_+^4 \mu_+,
	\end{align*}
	which completes the proof of Lemma \ref{FPKAMYINLI8}.
\end{proof}

Combining the above lemmas, we  complete a full step in our KAM iteration.

\section{Proof of the main KAM results}\label{FPKAMsec-5}

\subsection{Proof of Theorem \ref{FPKAMT1}}\label{FPKAMProofT1}
	\subsubsection{The iteration lemma}
In this section, we establish an iteration lemma which guarantees the inductive construction of the transformations in all the KAM steps.

Let $ {r_0},{s_0},{\gamma _0},{\mu _0},{H_0},{e_0},{{\bar h}_0},{P_0} $ be given at the beginning of Section \ref{FPKAMsec-4}, and set $ {D_0} = {D_0}\left( {{s_0},{r_0}} \right),{K_0} = 0,{\Phi _0} = \rm{id} $. We define the following sequence inductively for all $ \nu  \geqslant 1 $:
\begin{align}
	{r_\nu } ={}& {r_0}\left( {1 - \sum\limits_{i = 1}^\nu  {{2^{ - i - 1}}} } \right),\notag \\
	{s_\nu } ={}& \frac{1}{8}{\alpha _{\nu  - 1}}{s_{\nu  - 1}},\notag \\
	{\alpha _\nu } ={}& \mu _\nu ^{2\rho } = \mu _\nu ^{\frac{1 }{5}},\notag \\
	{\mu _\nu } ={}& {8^4}{c_0}\mu _{\nu  - 1}^{1 + \rho },\notag \\
	{K_\nu } ={}& {\left( {\left[ { - \ln {\mu _{\nu  - 1}}} \right] + 1} \right)^{3\eta }},\notag \\
	{{\tilde D}_\nu } ={}& D\left( {\frac{1}{2}{s_\nu },{r_\nu } + \frac{3}{4}\left( {{r_{\nu  - 1}} - {r_\nu }} \right)} \right).\notag
\end{align}

\begin{lemma}\label{Iteration lemma}
	Denote $ {\mu _ * } = {\mu _0}/\left( {{{\left( {{M^ * } + 2} \right)}^{3}}K_1^{5\left( {\tau  + 1} \right)}} \right) $. If  $ \varepsilon >0 $ is sufficiently small, then the KAM step described as above is valid for all $ \nu  \geqslant 0 $, resulting in sequences $ {H_\nu },{N_\nu },{e_\nu },{{\bar h}_\nu },{P_\nu },{\Phi _\nu } $ for $ \nu  \geqslant 1 $ with the following estimates:
	\begin{align}
	\label{FPKAM91}	&\left| {{e_{\nu  + 1}} - {e_\nu }} \right|,\;{\left\| {{{\bar h}_{\nu  + 1}} - {{\bar h}_\nu }} \right\|_{D\left( {{s_\nu }} \right)}},\;{\left\| {{P_\nu }} \right\|_{D\left( {{s_\nu },{r_\nu }} \right)}},\;\left| {{\xi _{\nu  + 1}} - {\xi _\nu }} \right| \leqslant \mu _ * ^{\frac{1}{2}}{2^{ - \nu }},\\
			&\left| {{e_\nu } - {e_0}} \right|,\;{\left\| {{{\bar h}_\nu } - {{\bar h}_0}} \right\|_{D\left( {{s_\nu }} \right)}} \leqslant 2\mu _ * ^{\frac{1}{2}}. \notag 
	\end{align}
	In addition, $ {\Phi _{\nu  + 1}}:{{\tilde D}_{\nu  + 1}} \to {{\tilde D}_\nu } $ is symplectic, and
	\begin{equation}\label{FPKAMdafai}
		{\left\| {{\Phi _{\nu  + 1}} - {\rm{id}}} \right\|_{{{\tilde D}_{\nu  + 1}}}} \leqslant \mu _ * ^{\frac{1}{2}}{2^{ - \nu }}.
	\end{equation}
	Moreover, it holds on $ {D_{\nu  + 1}} $ that
	\begin{equation}\label{FPKAMdaH}
		{H_{\nu  + 1}} = {H_\nu } \circ {\Phi _{\nu  + 1}} = {N_{\nu  + 1}} + {P_{\nu  + 1}}.
	\end{equation}
\end{lemma}
\begin{proof}
	The proof of this iterative lemma is based on verifying the assumptions in Lemmas \ref{FPKAM3.2} to \ref{FPKAMYINLI8}, i.e., verifying them for all fixed $\nu$. Without loss of generality, let $r_0 = 1$. Since $\varepsilon>0$ is sufficiently small, $\mu_0$ is also sufficiently small. Therefore, the first inequality in \eqref{FPKAMYINLI3} of Lemma \ref{FPKAM3.3}, and the assumptions in Lemmas \ref{FPKAMcrucial} and \ref{FPKAMYINLI8} all hold for $\nu = 0$. Moreover, by \eqref{FPKAMCANSHU}, the second inequality in \eqref{FPKAMYINLI3} of Lemma \ref{FPKAM3.3} also holds for $\nu = 0$. According to the definition of $\mu_\nu$, it is evident to calculate that
	\begin{equation}\label{eq44}
		\mu_\nu = (8^4 c_0)^{\frac{(1+\rho)^\nu - 1}{\rho}} \mu_0^{(1+\rho)^\nu}.
	\end{equation}
Fix $\zeta \gg 1$ and $0 < \mu_0 \ll 1$ such that the following holds:
\begin{equation}\label{eq29}
	\mu_0 < \left( \frac{1}{8^4 c_0 \zeta} \right)^{\frac{1}{\rho}} < 1.
\end{equation}
It then follows that
\begin{align}
	\mu_1 &= 8^4 c_0 \mu_0^{1+\rho} < \frac{1}{\zeta} \mu_0 < 1, \notag \\
	\mu_2 &= 8^4 c_0 \mu_1^{1+\rho} < \frac{1}{\zeta} \mu_1 < \frac{1}{\zeta^2} \mu_0, \notag \\
	&\vdots \notag \\
	\mu_\nu &= 8^4 c_0 \mu_{\nu-1}^{1+\rho} < \ldots < \frac{1}{\zeta^\nu} \mu_0. \label{eq28}
\end{align}
Let $ 	\Gamma_\nu = \Gamma(r_\nu - r_{\nu+1})$.
Note the relation
\begin{equation}\label{eq26}
	\frac{r_\nu - r_{\nu+1}}{r_0} = \frac{1}{2^{\nu+2}}. % = \frac{\beta_\nu - \beta_{\nu+1}}{\beta_0}.
\end{equation}
Since
\begin{align*}
	\Gamma_\nu  \leqslant \int_1^\infty t^{6\tau+7} e^{-\frac{t}{2^{\nu+5}}}   dt  \leqslant (6\tau + 7)! 2^{(\nu+5)(6\tau+7)},
\end{align*}
it is evident that if $\zeta$ is sufficiently large, then
\begin{align*}
	\mu_\nu^\rho \Gamma_\nu^i < \mu_\nu^\rho (\Gamma_\nu^i + \Gamma_\nu) \leqslant 1,  \quad  i = 1, 2,
\end{align*}
which implies that \eqref{FPKAMYINLI8JIASHE} in Lemma \ref{FPKAMYINLI8} holds for all $\nu \geq 1$, and
\begin{equation}\label{eq27}
	\mu_\nu \Gamma_\nu \leqslant \mu_\nu^{1-\rho} \leqslant \frac{\mu_0^{1-\rho}}{\zeta^{(1-\rho)\nu}}.
\end{equation}
Based on \eqref{eq26} and \eqref{eq27}, it is evident to verify that \eqref{FPKAMYINLI71} in Lemma \ref{FPKAMYINLI7} holds for all $\nu \geqslant 1$, whenever $\zeta \gg 1$ and $0 < \mu_0 \ll 1$.

By the second inequality in \eqref{FPKAMR} of Lemma \ref{FPKAM3.2} and \eqref{eq28}, we have
\begin{align*}
	{\mathop {\sup }\limits_{\xi  \in \mathcal{O}} }	\left\|\sum_{j=0}^{\nu} p_{01}^j\right\|_{D(s_\nu, r_\nu)} < c \sum_{j=0}^{\nu} \mu_j < c \sum_{j=0}^{\nu} \frac{1}{\zeta^j} \mu_0 < c \mu_0^{\frac{1}{2}},
\end{align*}
which verifies the assumption \eqref{FPKAMYINLI4JIASHE} in Lemma \ref{FPKAMcrucial}.

To verify the second inequality in \eqref{FPKAMYINLI3} of Lemma \ref{FPKAM3.3}, we note that when $\zeta \gg 1$, \eqref{eq44} and \eqref{eq28} imply
\begin{equation*}
	\frac{1}{4}\left(M^* + 2\right) \mu_{\nu-1}^{2\rho} K_{\nu+1}^{\tau+1} < \frac{1}{2^{\nu+2}}.
\end{equation*}
Thus,
\begin{align}
	2\left(M^* + 2\right) s_\nu K_{\nu+1}^{\tau+1} \leqslant \frac{s_{\nu-1}}{4}\left(M^* + 1\right) \mu_{\nu-1}^{2\rho} K_{\nu+1}^{\tau+1} \leqslant \frac{s_0}{2^{\nu+2}} < \frac{\gamma_0}{2^{\nu+2}} < \gamma_0, \notag
\end{align}
which verifies the second inequality in \eqref{FPKAMYINLI3} of Lemma \ref{FPKAM3.3} for all $\nu \geqslant 1$.

Note that \eqref{eq29} and \eqref{eq28} imply $\zeta^{1 - \rho} \geqslant 2$. Thus, for any $\nu \geqslant 1$, we have
\begin{align}
	c_0 \mu_\nu &\leqslant \frac{\mu_0}{2^\nu} \leqslant \frac{\mu_*^{\frac{1}{2}}}{2^\nu}, \label{eq30} \\
	c_0 \mu_\nu \Gamma_\nu &\leqslant \frac{\mu_0^{1 - \rho}}{2^\nu} \leqslant \frac{\mu_*^{\frac{1}{2}}}{2^\nu}, \notag  \\
	c_0 s_\nu^{3} \mu_\nu &\leqslant \frac{\mu_0^{1 + 6\rho} s_0^{3}}{2^{\nu + 3}} \leqslant \frac{\mu_*}{2^\nu}. \label{eq37}
\end{align}
By \eqref{eq30} and the induction in \eqref{FPKAMYINLI53} of Lemma \ref{FPKAMYINLI5}, the second inequality in \eqref{FPKAMYINLI3} of Lemma \ref{FPKAM3.3} holds for all $\nu \geqslant 1$.

Since $(1 + \rho)^\eta > 2$, we have
\begin{align*}
	\frac{1}{2^{\nu + 6}} \left( \left[ -\ln \mu \right] + 1 \right)^\eta &\geqslant \frac{1}{2^{\nu + 6}} \left( \left(1 - (1 + \rho)^\nu \right) \ln \left(8^4 c_0 \right) - (1 + \rho)^\nu \ln \mu_0 \right)^\eta \\
	&\geqslant \frac{1}{2^{\nu + 6}} (1 + \rho)^{\eta \nu} \left(- \ln \mu_0 \right)^\eta \geqslant 1.
\end{align*}
Therefore, when $\mu$ is sufficiently small (which is guaranteed by $\varepsilon$ being sufficiently small), we have
	\begin{align*}
	&\ln\left(n+1\right)!+\left(\nu+6\right)n\ln2+3n\eta\ln\left(\left[-\ln\mu\right]+1\right)-\frac{1}{2^{\nu+6}}\left(\left[-\ln\mu\right]+1\right)^{3\eta}\\
	\leqslant&\ln\left(n+1\right)!+\left(\nu+6\right)n\ln2+3n\eta\ln\left(-\ln\mu+2\right)-\left(-\ln\mu\right)^{2\eta}\\
	\leqslant&-\ln\frac{1}{\mu}.
\end{align*}
It follows that
\begin{equation*}
	\int_{K_{\nu + 1}}^\infty t^n e^{-\frac{t}{2^{\nu + 6}}} dt \leqslant (n + 1)! 2^{(\nu + 6)n} K_{\nu + 1}^n e^{-\frac{K_{\nu + 1}}{2^{\nu + 6}}} \leqslant \mu,
\end{equation*}
which means that the assumption \eqref{FPKAMH1} in Lemma \ref{FPKAM3.2} holds for all $\nu \geqslant 1$.

In summary, the KAM step in Section \ref{FPKAMsec-4} holds for all $\nu \geqslant 0$, and therefore the conclusions of all lemmas hold.

Under these grounds, the first and second inequalities in \eqref{FPKAM91} of Lemma \ref{Iteration lemma} are proved by Lemma \ref{FPKAMYINLI5}, \eqref{eq30}, and \eqref{eq37}; the third inequality in \eqref{FPKAM91} is proved by \eqref{FPKAMyinli83} in Lemma \ref{FPKAMYINLI8} and \eqref{eq30}; the fourth inequality in \eqref{FPKAM91} is proved by \eqref{FPKAMYINLI5} in Lemma \ref{FPKAMYINLI5} and \eqref{eq30}; and finally, \eqref{FPKAMdafai} and \eqref{FPKAMdaH} are proved by Lemma \ref{FPKAMYINLI7}. Thus, we  complete the proof of Lemma \ref{Iteration lemma}.
\end{proof}

\subsubsection{Convergence}
The convergence is standard in KAM theory. For the sake of completeness, we briefly give the framework of the proof. Let
\[{\Psi ^\nu }: = {\Phi _1} \circ {\Phi _2} \circ  \cdots {\Phi _\nu }, \quad \nu  \geqslant 1.\]
Then by Lemma \ref{Iteration lemma}, we have
\[{D_{\nu  + 1}} \subset {D_\nu }, \quad {\Psi ^\nu }:{{\tilde D}_\nu } \to {{\tilde D}_0}, \quad {H_0} \circ {\Psi ^\nu } = {H_\nu } = {N_\nu } + {P_\nu }\]
and
\begin{equation}\label{FPKAMN+}
	{N_ + } = {e_\nu } + \left\langle {\omega \left( {{\xi _\nu }} \right) + \sum\limits_{j = 0}^\nu  {p_{01}^j\left( {{\xi _\nu }} \right)} ,y} \right\rangle  + {{\bar h}_\nu }\left( {y,{\xi _\nu }} \right), \quad \nu  \geqslant 0,
\end{equation}
where $ {\Psi ^0} = {\rm{id}} $. Using \eqref{FPKAMdafai} and the identity
\[{\Psi ^\nu } =  {\rm{id}} + \sum\limits_{j = 0}^\nu  {\left( {{\Psi ^j} - {\Psi ^{j - 1}}} \right)} ,\]
we can verify that $ {\Psi ^\nu } $ is uniformly convergent and denote the limit by $ {\Psi ^\infty } $.

In view of Lemma \ref{Iteration lemma} and Lemma \ref{FPKAMcrucial}, it is clear to see that $ {e_\nu },{{\bar h}_\nu }$ and ${\xi _\nu } $ converge uniformly with respect to $ \nu $, and we denote their limits by $ {e_\infty },{{\bar h}_\infty },\xi _\infty:=\xi^* $, respectively. In addition, Lemma \ref{FPKAMcrucial} gives the following frequency-preserving equations:
\begin{align}
	{}&\omega \left( {{\xi _1}} \right) + p_{01}^0\left( {{\xi _1}} \right) = \omega \left( {{\xi _0}} \right),\notag \\
	{}&\omega \left( {{\xi _2}} \right) + p_{01}^0\left( {{\xi _2}} \right) + p_{01}^1\left( {{\xi _2}} \right) = \omega \left( {{\xi _0}} \right),\notag \\
	{}&\vdots\notag \\
	\label{FPKAMpinlvlie}{}&\omega \left( {{\xi _\nu }} \right) + p_{01}^0\left( {{\xi _\nu }} \right) +  \cdots  + p_{01}^{\nu  - 1}\left( {{\xi _\nu }} \right) = \omega \left( {{\xi _0}} \right).
\end{align}
Using the Cauchy property of $ \{\xi_\nu\}_{\nu \in \mathbb{N}^+} $ and taking limits at both sides of \eqref{FPKAMpinlvlie}, we get
\[\omega \left( {{\xi _\infty }} \right) + \sum\limits_{j = 0}^\infty  {p_{01}^j\left( {{\xi _\infty }} \right)}  = \omega \left( {{\xi _0}} \right).\]
Then on $ D\left( {{s_0}/2} \right) $,  we conclude from \eqref{FPKAMN+} that $ N_\nu $ converges uniformly to
\[{N_\infty } = {e_\infty } + \left\langle {\omega \left( {{\xi _0}} \right),y} \right\rangle  + {{\bar h}_\infty }\left( {y,{\xi _\infty }} \right).\]
Hence, \ $ {P_\nu } = {H_0} \circ {\Psi ^\nu } - {N_\nu } $ converges uniformly to $ {P_\infty } = {H_0} \circ {\Psi ^\infty } - {N_\infty } $ on $ D\left( {{s_0}/2,{r_0}/2} \right) $.
Since $ {\left\| {{P_\nu }} \right\|_{{D_\nu }}} \leqslant c\gamma _0^{n+6}s_\nu ^4{\mu _\nu } $, we have that $P_\nu$ converges to $ 0 $ as $ \nu  \to \infty  $, and $ J\nabla {P_\infty } = 0 $ on $ D\left( {0,{r_0}/2} \right) $.
Thus, for  $ {\xi _0} \in \mathcal{O}^o $ given in advance (see the Internal Condition \ref{aph1}), the Hamiltonian $ {H_\infty } = {N_\infty } + {P_\infty } $ admits an analytic, quasi-periodic, invariant $ n $-torus $\left\{ 0 \right\}\times {\mathbb{T}^n}   $ with the prescribed  Diophantine frequency $ \omega \left( {{\xi _0}} \right) $. Moreover, the asymptotic property for the parametric limit  $ \xi_\infty $, i.e., $ \left| {{\xi _\infty } - {\xi _0}} \right| =  \mathcal{O}\left( \varphi \left( \varepsilon  \right){ - \int_0^{\varepsilon } {\frac{{\varphi (x)}}{{x\ln x}}dx} } \right) = o\left( 1 \right) $ as $ \varepsilon \to 0^+ $, has been proved in Lemma \ref{FPKAMcrucial}.
\vspace{5mm}

Next, we prove the second part of Theoreom \ref{FPKAMT1}. For $ \vartheta:=\min_{\xi\in\partial \mathcal{V}}|\omega(\xi)-\omega(\xi_0)|$, define a set $ \mathscr{V} $ as 
\[\mathscr{V}:=\left\{\hat\xi_0\in \mathcal{V}:|\omega(\hat \xi_0)-\omega(\xi_0)|<\vartheta,\;|\langle k,\omega(\hat\xi_0)\rangle|>{\gamma}{|k|^{-\tau}},\;\forall 0 \ne k\in\mathbb{Z}^n\right\}.\]
Then for any  $\hat\xi_0,\tilde\xi_0\in\mathscr{V}$, we can obtain similar to  \eqref{FPKAMpinlvlie} that
\[	\omega(\xi_\infty(\hat\xi_0))+\sum_{j=0}^{\infty}p_{01}^j(\xi_\infty(\hat\xi_0))=\omega(\hat\xi_0), \quad 
\omega(\xi_\infty(\tilde\xi_0))+\sum_{j=0}^{\infty}p_{01}^j(\xi_\infty(\tilde\xi_0))=\omega(\tilde\xi_0),\]
and this leads to 
	\begin{align*}
		\omega(\xi_\infty(\hat\xi_0))-\omega(\xi_\infty(\tilde\xi_0))+\sum_{j=0}^{\infty}p_{01}^j(\xi_\infty(\hat\xi_0))-p_{01}^j(\xi_\infty(\tilde\xi_0))=\omega(\hat\xi_0)-\omega(\tilde\xi_0).
	\end{align*}
As a consequence, with the Relative Singularity Condition \ref{aph2} and the smallness of $ \varepsilon>0 $, we have
	\begin{align*}
		\left|\omega(\hat\xi_0)-\omega(\tilde\xi_0)\right|
	&	=\left|\omega(\xi_\infty(\hat\xi_0))-\omega(\xi_\infty(\tilde\xi_0))+\sum_{j=0}^{\infty}p_{01}^j(\xi_\infty(\hat\xi_0))-p_{01}^j(\xi_\infty(\tilde\xi_0))\right|\\
	&	\geqslant \left|\omega(\xi_\infty(\hat\xi_0))-\omega(\xi_\infty(\tilde\xi_0))\right|-\sum_{j=0}^{\infty}\left|p_{01}^j(\xi_\infty(\hat\xi_0))-p_{01}^j(\xi_\infty(\tilde\xi_0))\right |\\ &\geqslant\left|\omega(\xi_\infty(\hat\xi_0))-\omega(\xi_\infty(\tilde\xi_0))\right|-\left(c\sum_{j=0}^{\infty}\mu_j \right) \left|\omega(\xi_\infty(\hat\xi_0))-\omega(\xi_\infty(\tilde\xi_0))\right|\notag\\
&\geqslant\frac{1}{2}\left|\omega(\xi_\infty(\hat\xi_0))-\omega(\xi_\infty(\tilde\xi_0))\right|,
	\end{align*}
hence, by the Controllability Condition \ref{aph3} we arrive at the continuity of $\xi_\infty(\hat\xi_0)$ with respect to  $\hat\xi_0\in \mathscr{V}$ as (note that $ \varphi $ and $ \omega $ are both continuous, and $ \varphi(0)=0 $)
	\begin{align}  \label{FPKAMLIP}
	\left\vert\xi_\infty(\hat\xi_0)-\xi_\infty(\tilde\xi_0)\right\vert \leqslant \varphi \left(2 \left|\omega(\hat\xi_0)-\omega(\tilde\xi_0)\right|\right).
	\end{align}
In particular,  we could estimate explicitly the modulus of continuity of $\xi_\infty(\hat\xi_0)$ from \eqref{FPKAMLIP} and Comment \ref{comc2}. Denote by $ \varpi_+ (x)$ and $ \varpi_- (x) $ the modulus of continuity of  $ \omega(\xi) $ and its inverse $ \omega^{-1}(\xi) $, respectively. If $ \varpi_+ (x)$ is weaker than the Lipschitz type $ x $, then $ \varpi_- (x) =\mathcal{O}^{\#} (x) $, and $ \varphi(x) $ could be $ \mathcal{O}^{\#} (x)  $ by Comment \ref{comc2}. In this case, \eqref{FPKAMLIP} gives the estimate 
\[	\left\vert\xi_\infty(\hat\xi_0)-\xi_\infty(\tilde\xi_0)\right\vert \lesssim \varpi_+ \left(\vert\hat\xi_0-\tilde\xi_0\vert\right).\]
In a similar way, if  $ \varpi_- (x)$ is weaker than the Lipschitz  type, then $ \varpi_+ (x) =\mathcal{O}^{\#} (x) $,  and $ \varphi(x) $ could be $ \mathcal{O}^{\#} (\varpi_- (x))  $, which leads to 
\[	\left\vert\xi_\infty(\hat\xi_0)-\xi_\infty(\tilde\xi_0)\right\vert \lesssim \varpi_- \left(\vert\hat\xi_0-\tilde\xi_0\vert\right).\]
Combining the above two cases, the modulus of continuity $ \varpi(x) $ of $\xi_\infty(\hat\xi_0)$ is the weaker one of  $ \varpi_+ (x)$ and $ \varpi_ - (x)$. Finally, applying the previous process and replacing $\xi_0$ with $\hat\xi_0$, we obtain an analytic, quasi-periodic, invariant $ n $-torus with the  Diophantine frequency $\omega(\hat\xi_0)$.	This completes the proof of Theorem \ref{FPKAMT1}.

 \subsection{Proof of Corollary \ref{FPKAMCORO1}}
In view of the comments below the Controllability Condition \ref{aph3}, we deduce from the Relative Singularity Condition \ref{aph2} that the function $ \varphi(\delta) = \mathcal{O}(\delta^{{\alpha}^{-1}})  $ as $ \delta\to 0^+ $ without loss of generality. Then the Controllability Condition \ref{aph3}  automatically holds for any $ 0<{\tau'}\leqslant1/2 $:
 \[ - \int_0^{\tau'}  {\frac{{\varphi (x)}}{{x\ln x}}dx}  \lesssim  - \int_0^{\tau'}  {\frac{1}{{{x^{1 - \alpha^{-1} }}\ln x}}dx}  \lesssim \int_0^{\tau'}  {\frac{1}{{{x^{1 - \alpha^{-1} }}}}dx}  <  + \infty. \]
 Therefore, by applying Theorem \ref{FPKAMT1}, we prove this corollary directly.
 
 %\section{Appendix}

 \subsection{Proof of Corollary \ref{FPKAMCORO2}}
 Note that for a non-degenerate linear frequency mapping $ \omega(\xi) $ defined on the parameter set $ \mathcal{O}\subset \mathbb{R}^n $, almost all frequencies in $ (\omega(\mathcal{O}))^o $ are $ \tau $-Diophantine with $ \tau>\max\{n-1,1\} $, i.e., they satisfy \eqref{FPKAMDIO}, and the corresponding parameters belong to $ \mathcal{O}^o $. Therefore,  recalling Comment (C2) and the Lipschitz regularity for the  perturbation, we prove the desired conclusion by directly applying Theorem \ref{FPKAMT1}.

  \section*{Acknowledgements} 
Z. Tong  was supported by the National Natural Science Foundation of China (Grant No. 12601347) and the China Postdoctoral Science Foundation (Grant No. 2025M783102).
Y. Li was supported in part by the National Natural Science Foundation of China (Grant Nos. 12471183 and 12531009).  The authors would also like to sincerely thank the anonymous referees for their valuable suggestions and comments which led to a significant improvement of the paper.

\end{document}